\documentclass[preprint]{imsart}

\usepackage{natbib}
\usepackage[T1]{fontenc}
\usepackage[latin1]{inputenc}
\usepackage{amsmath,amsthm,amssymb,amsfonts,epic,latexsym,hyperref}
\usepackage{graphicx}
\usepackage[usenames]{color}

\usepackage{lipsum}
\usepackage{rotating}
\usepackage{dsfont}
\usepackage{geometry}
\usepackage{stmaryrd}
\usepackage{ifthen}
\usepackage[nottoc]{tocbibind}
\usepackage{verbatim}
\usepackage{tabularx}
\usepackage{framed}
\usepackage{longtable}
\usepackage{tabu}
\usepackage{tikz}
\allowdisplaybreaks

\newcommand\un{\mathds{1}}
\newcommand\eps{\varepsilon}
\renewcommand\phi{\varphi}

\newcommand\esp[1]{\mathbb{E}\left[#1\right]}

\newcommand\espcc[2]{\mathbb{E}_{#1}\left[#2\right]}
\newcommand\espccc[3]{\mathbb{E}_{#1}\left[#3\right]}

\newcommand\uno[1]{\un_{\left\{#1\right\}}}

\newcommand\mesinv{\lambda}

\newcounter{comptetape}
\newtheorem{thm}{Theorem}[section]
\newtheorem{prop}[thm]{Proposition}

\newtheorem{assu}{Assumption}
\newtheorem{lem}[thm]{Lemma}
\newtheorem{rem}[thm]{Remark}
\newtheorem{defi}[thm]{Definition}

\newtheorem{example}[thm]{Example}

\newcommand\deriv[2]{\frac{\partial\ifthenelse{\equal{#1}1}{}{^{#1}} \bar{X}_{#2}^{(x)}}{\partial x\ifthenelse{\equal{#1}1}{}{^{#1}}}}
\newcommand\derivf[1]{\left(\sqrt{f}\right)^{\ifthenelse{\equal{#1}{1}}{'}{(#1)}}}

\newcommand\D{\mathcal {D}}
\newcommand\LL{\mathbb{L}}

\newcommand\n{\mathbb{N}}
\renewcommand\r{\mathbb{R}}
\newcommand\R{\r}
\newcommand\E{\mathbb{E}}

\newcommand\constantetruc{C}
\renewcommand\ll{\left}
\newcommand\rr{\right}

\arxiv{arxiv:1904.06985}

\begin{document}

\begin{frontmatter}

\title{Mean field limits for interacting Hawkes processes in a diffusive regime}

\runtitle{Hawkes with random jumps}

\begin{aug}
 % indicate corresponding author with \corref{}

\author{\fnms{Xavier} \snm{Erny}\thanksref{a,e1}\corref{}\ead[label=e1,mark]{xavier.erny@univ-evry.fr}},
\author{\fnms{Eva}
\snm{L\"ocherbach}\thanksref{b,e2}\ead[label=e2,mark]{eva.locherbach@univ-paris1.fr}}
\and \author{\fnms{Dasha} \snm{Loukianova}\thanksref{a,e3} \ead[label=e3,mark]{dasha.loukianova@univ-evry.fr}}

\address[a]{Universit\'e Paris-Saclay, CNRS, Univ Evry, Laboratoire de Math\'ematiques et Mod\'elisation d'Evry, 91037, Evry, France\\\printead{e1,e3}}
\address[b]{Statistique, Analyse et Mod\'elisation Multidisciplinaire, Universit\'e Paris 1 Panth\'eon-Sorbonne, EA 4543 et FR FP2M 2036 CNRS\\\printead{e2}}

\runauthor{X. Erny et al.}

 \end{aug}

\begin{abstract}
We consider a sequence of systems of Hawkes processes having mean field interactions in a diffusive regime. The stochastic intensity of each process is a solution of a stochastic differential equation driven by $N$ independent Poisson random measures. We show that, as the number of interacting components $N$ tends to infinity, this intensity converges in distribution in the Skorokhod space to a CIR-type diffusion. Moreover, we prove the convergence in distribution of the Hawkes processes to the limit point process having the limit diffusion as intensity. To prove the convergence results, we use analytical technics based on the convergence of the associated infinitesimal generators and Markovian semigroups.
\end{abstract}

% 62M05 Markov processes: estimation
% 62F12 Asymptotic properties of estimators
% 60J25 Markov processes with continuous parameter
% 60J27 Markov chains with continuous parameter
% 60J35 Transition functions, generators and resolvents


\end{frontmatter}

\renewcommand\cite[1]{\citep{#1}}

\section*{Introduction}

Hawkes processes were originally introduced by \cite{hawkes} to model the appearance of earthquakes in Japan. Since then these processes have been successfully used in many fields to model various physical, biological or economical phenomena exhibiting self-excitation or -inhibition and interactions, such as seismology (\cite{helmstetter}, \cite{kagan}, \cite{ogata}, \cite{bacrymuzy}),  financial contagion (\cite{aitsahalia}), high frequency financial order books arrivals (\cite{abergel}, \cite{bauwens}, \cite{hewlett}), genome analysis (\cite{reynaudbouret}) and interactions in social networks (\cite{zhou}). In particular, multivariate Hawkes processes are extensively used in neuroscience to model temporal arrival of spikes in neural networks (\cite{grun}, \cite{neuro}, \cite{pillow}, \cite{reynaudbouret2}) since they provide good models to describe the typical temporal decorrelations present in spike trains of the neurons as well as the functional connectivity in neural nets.

In this paper, we consider a sequence of multivariate Hawkes processes $(Z^N)_{N\in\n^*}$ of the form $Z^N=(Z^{N,1}_t,\ldots Z^{N,N}_t)_{t\geq 0}$. Each $Z^N$ is designed to describe the behaviour of some interacting system with $N$ components,  for example a neural network of $N$ neurons. More precisely, $ Z^N$ is a multivariate counting process where each $Z^{N,i}$ records the number of events related to the $i-$th component, as for example the number of spikes of the $i-$th neuron. These counting processes are interacting, that is, any event of type $i$ is able to trigger or to inhibit future events of all other types $j$. The process $(Z^{N,1},\hdots,Z^{N,N})$ is informally defined via its stochastic intensity process $\lambda^{N}=(\lambda^{N,1}(t),\hdots,\lambda^{N,N}(t))_{t\geq 0}$ through the relation

$$ \mathbb{P} ( Z^{N,i}\mbox{ has a jump in  } ]t, t + d t] | {\cal F}_t ) = \lambda^{N,i} ( t) dt , ~~1\leq i\leq N,$$
where $\mathcal{F}_t=\sigma\left(Z^N_s~:~0\leq s\leq t\right).$ The stochastic intensity of a Hawkes process is given by
\begin{equation}
\label{intensitehawkes}
 \lambda^{N,i} ( t ) = f^N_i \left(\sum_{ j = 1 }^N \int_{ -\infty  }^t h_{ij}^N(t-s)  d Z^{N,j} ( s) \right) .
\end{equation} 
Here, $h_{ij}^N$ models the action or the influence of events of type $j$ on those of type $i$, and how this influence decreases as time goes by. The function $f^N_i$ is called the jump rate function of $Z^{N,i}$.

Since the founding works of~\cite{hawkes} and~\cite{hawkesoakes}, many probabilistic properties of Hawkes processes have been well-understood, such as ergodicity, stationarity and long time behaviour (see \cite{bremaud}, \cite{daley}, \cite{costa}, \cite{raad_renewal_2019} and \cite{graham_regenerative_2019}). A number of authors studied the statistical inference for Hawkes processes (\cite{ogata2} and~\cite{reynaudbouret}). Another field of study, very active nowadays, concerns the behaviour of the Hawkes process when the number of components $N$ goes to infinity. During the last decade, large population limits of systems of interacting Hawkes processes have been studied in \cite{nicolaseva},~\cite{highdimensional} and \cite{evasusanne}.

In \cite{highdimensional}, the authors consider a general class of Hawkes processes whose interactions are given by a graph. In the case where the interactions are of mean field type and scaled in $N^{-1}$, namely $h_{ij}^N=N^{-1}h$ and $f^N_i=f$ in~\eqref{intensitehawkes}, they show that the Hawkes processes can be approximated by an i.i.d. family of inhomogeneous Poisson processes. They observe that for each fixed integer $k$, the joint law of $k$ components converges to a product law as $N$ tends to infinity, which is commonly referred to as the propagation of chaos. \cite{evasusanne} generalize this result to a multi-population frame and show how oscillations emerge in the large population limit. Note again that the interactions in both papers are scaled in $N^{-1}$, which leads to limit point processes with deterministic intensity.

The purpose of this paper is to study the large population limit (when $N$ goes to infinity) of the multivariate Hawkes processes $(Z^{N,1},\hdots,Z^{N,N})$ with mean field interactions scaled in $N^{-1/2}$. Contrarily to the situation considered in~\cite{highdimensional} and~\cite{evasusanne}, this scaling leads to a non-chaotic limiting process with stochastic intensity. As we consider interactions scaled in $N^{-1/2}$, we have to center the terms of the sum in~\eqref{intensitehawkes} to make the intensity process converge according to some kind of central limit theorem. To this end, we consider intensities with stochastic jump heights. Namely, in this model, the multivariate Hawkes processes $(Z^{N,i})_{1\leq i\leq N}$ ($N\in\n^*$) are of the form
\begin{equation}
\label{ZNit}
Z^{N,i}_t=\int_{]0,t]\times\r_+\times\r}\uno{z\leq\lambda^N_s}d\pi_i(s,z,u),~~1 \le i \le N, 
\end{equation}
where $(\pi_i)_{i\in\n^*}$ are i.i.d. Poisson random measures on $\r_+\times\r_+\times\r$ of intensity $ds\, dz\, d\mu(u)$ and $\mu$ is a centered probability measure on~$\r$ having a finite second moment $\sigma^2$. The stochastic intensity of $Z^{N,i}$ is given by
$$\lambda^{N,i}_t=\lambda^N_t=f\left(X^N_{t-}\right),$$
where
$$X_t^N=\frac{1}{\sqrt{N}}\sum_{j=1}^N \int_{[0,t]\times\r_+\times\r}h(t-s)u\uno{z\leq f\left(X^N_{s-}\right)}d\pi_j(s,z,u).$$

Moreover we consider a function $h$ of the form $h(t)=e^{-\alpha t}$ so that the process $(X_t^N)_t$ is a piecewise deterministic Markov process. In the framework of neurosciences, $X_t^N$ represents the membrane potential of the neurons at time $t . $ The random jump heights $u ,$ chosen according to the measure $ \mu,$ model random synaptic weights and the jumps of $Z^{N,j}$ represent the spike times of neuron $j$. If neuron $j$ spikes at time $t$, an additional random potential height $u /\sqrt{N}$ is given to all other neurons in the system. As a consequence, the process $X^N$ has the following dynamic
\begin{equation*}
\begin{array}{l}dX_t^N=-\alpha X_t^Ndt+\frac{1}{\sqrt{N}}\displaystyle\sum_{j=1}^N\int_{\r_+\times\r}u\uno{z\leq f\left(X_{t-}^N\right)}d\pi_j(t,z,u).
\end{array}
\end{equation*}
Its infinitesimal generator is given by
$$ A^N g (x) = - \alpha x\,  g' (x) + N f(x)   \int_\r \left[ g \left( x + \frac{u}{\sqrt{N}} \right) - g(x) \right] \mu ( du ),$$
for sufficiently smooth functions $g$. As $N$ goes to infinity, the above expression converges to
$$ \bar{A} g( x) = - \alpha x\,  g' (x) + \frac{\sigma^2}{2} f(x) g'' ( x) , $$
which is the generator of a CIR-type diffusion given as solution of the SDE
\begin{equation}
\label{sde1}
d \bar{X}_t = - \alpha \bar{X}_t dt + \sigma\sqrt{ f ( \bar{X}_t)}  d B_t.
\end{equation}
It is classical to show in this framework that the convergence of generators implies the convergence of $X^N$ to $\bar{X}$ in distribution in the Skorokhod space. In this article we establish explicit bounds for the weak error for this convergence by means of a Trotter-Kato like formula. Moreover we establish for each $i,$ the convergence in distribution in the Skorokhod space of the associated counting process $Z^{N,i}$ to the limit counting process $\bar{Z}^i$ which has intensity $(f(\bar{X}_t))_t.$ Conditionally on $\bar{X}$, the $\bar{Z}^i, i \geq 1,$ are independent. This property can be viewed as a conditional propagation of chaos-property, which has to be compared to \cite{highdimensional} and \cite{evasusanne} where the intensity of the limit process is deterministic and its components are truly independent, and to \cite{delarue}, \cite{dawson_stochastic_1995} and \cite{kurtz_particle_1999} where all interacting components are subject to common noise. In our case, the common noise, that is, the Brownian motion $ B$ of \eqref{sde1}, emerges in the limit as a consequence of the central limit theorem. 

To obtain a precise control of the speed of convergence of $X^N$ to $\bar{X}$ we use analytical methods showing first the convergence of the generators from which we deduce the convergence of the semigroups via the formula 
\begin{equation}
\label{formulegenerateur}
\bar{P}_tg(x)-P^N_tg(x)=\int_0^t P_{t-s}^N\left(\bar{A}-A^N\right)\bar{P}_sg(x)ds.
\end{equation}
Here $\bar{P}_tg(x)=\espcc{x}{g(\bar{X}_t)}$ and $P_t^Ng(x)=\espccc{x}{N}{g(X_t^N)}$ denote the Markovian semigroups of $\bar{X}$ and $X^N$. This formula is well-known in the classical semigroup theory setting where the generators are strong derivatives of semigroups in the Banach space of continuous bounded functions (see Lemma~1.6.2 of~\cite{ethier}). In our case, we have to consider extended generators (see~\cite{davis} or~\cite{meyn}), i.e. $A^Ng(x)$ is the point-wise derivative of $t\mapsto P_t^Ng(x)$ in $0.$ {The proof of formula~\eqref{formulegenerateur} for our extended generators is given in the Appendix (Proposition~\ref{generatorsemigroup}).}

It is well-known that under suitable assumptions on $f,$ the solution of~\eqref{sde1} admits a unique invariant measure $\mesinv$ whose density is explicitly known. Thus, a natural question is to consider the limit of the law ${\mathcal L}(X_t^N)$ of $X^N_t$ when $t$ and $N$ go simultaneously to infinity. We prove that the limit of ${\mathcal L}(X_t^N)$ is $\lambda$, for $(N, t ) \to (\infty,\infty) ,$ under suitable conditions on the joint convergence of $(N,t)$. We also prove that there exists a parameter $ \alpha^*$ such that for all  $ \alpha > \alpha^*,$ this converges holds whenever $ (N, t ) \to (\infty, \infty ) $ jointly, without any further condition, and we provide a control of the error (Theorem~\ref{convergencesemigroupe2}).

The paper is organized as follows: in Section~\ref{notationassumptions}, we state the assumptions and formulate the main results. Section~\ref{convergencesemigroups} is devoted to the proof of the convergence of the semigroup of $X^N$ to that of $\bar{X}$ (Theorem~\ref{convergencesemigroupe1}$.(i)$), and Section~\ref{convergencetransition} to the study of the limit of the law of $X^N_t$ as $N, t \to \infty$ (Theorem~\ref{convergencesemigroupe2}). In Section~\ref{convergencepoint}, we prove the convergence of the systems of point processes $(Z^{N,i})_{1\leq i\leq N}$ to $(\bar{Z}^i)_{i\geq 1}$ (Theorem~\ref{convergenceZNZ}). Finally in the Appendix, we {collect some results about extended generators and we give the proof of \eqref{formulegenerateur} together with some other technical results that we use throughout the paper.}

\section{Notation, assumptions and main results}\label{notationassumptions}
\subsection{Notation}\label{notation}
The following notation are used throughout the paper:
\begin{itemize}
\item If $X$ is a random variable, we note $\mathcal{L}(X)$ its distribution.
\item If $g$ is a real-valued function which is $n$ times differentiable, we note $||g||_{n,\infty}=\sum_{k=0}^n||g^{(k)}||_\infty.$
\item If $g : \r \to \r $ is a real-valued measurable function and $ \lambda $ a measure on $ (\r, {\cal B}(\r)) $ such that $ g $ is integrable with respect to $ \lambda, $ we write $ \lambda (g) $ for $ \int_\r g d \lambda.$ 
\item We write $C_b^n(\r)$ for the set of the functions $g$ which are $n$ times continuously differentiable such that $||g||_{n,\infty}<~+\infty$, and we write for short $C_b(\r)$ instead of $C_b^0(\r).$ Finally,  $C^n(\r)$ denotes the set of $n$ times continuously differentiable functions that are not necessarily bounded nor have bounded derivates.
\item If $g$ is a real-valued function and $I$ is an interval, we note $||g||_{\infty,I}={\sup}_{x\in I}|g(x)|.$
\item We write $C^n_c(\r)$ for the set of functions that are $n$ times continuously differentiable and that have a compact support.
\item We write $D(\r_+,\r)$ for the Skorokhod space of c\`adl\`ag functions from $\r_+$ to $\r,$ endowed with the Skorokhod metric (see Chapter~3 Section~16 of~\cite{billingsley}), and $D(\r_+,\r_+)$ for this space restricted to non-negative functions.
%\item $\mathcal{M}^\#$ denotes the space of locally finite measures on $\r_+\times\r_+$ endowed with the topology of the weak convergence, and $\mathcal{N}^\#$ the subspace that contains only the simple point measures.
\item $\alpha$ is a positive constant, $L,\sigma$ and $m_k$ ($1\leq k\leq 4$)  are fixed parameters defined in Assumptions~\ref{hyp1f},~\ref{hyp1control} and~\ref{hyp2} below. Finally, we note $\constantetruc$ any arbitrary constant, so the value of $\constantetruc$ can change from line to line in an equation. Moreover, if $\constantetruc$ depends on some non-fixed parameter $\theta$, we write $\constantetruc_\theta$.
\end{itemize}

\subsection{Assumptions}\label{assumptions}

Let $X^N$ satisfy
\begin{equation}
\label{XtNdefinition}
\left\{\begin{array}{l}dX_t^N=-\alpha X_t^Ndt+\frac{1}{\sqrt{N}}\displaystyle\sum_{j=1}^N\int_{\r_+\times\r}u\uno{z\leq f\left(X_{t-}^N\right)}d\pi_j(t,z,u),\\
X_0^N\sim\nu_0^N,\end{array}\right.
\end{equation}
where $\nu_0^N$ is a probability measure on $\r$. Under natural assumptions on $f,$ the SDE~\eqref{XtNdefinition} admits a unique non-exploding strong solution (see Proposition~\ref{wellposed}).

The aim of this paper is to provide explicit bounds for the convergence of $X^N$ in the Skorokhod space to the  limit process $(\bar{X}_t)_{t\in\r_+}$ which is solution to the SDE
\begin{equation}
\label{sde}
\left\{\begin{array}{l}d\bar{X}_t=-\alpha \bar{X}_tdt+\sigma\sqrt{f\left(\bar{X}_t\right)}dB_t,\\
\bar{X}_0\sim\bar{\nu}_0,\end{array}\right.
\end{equation}
where $\sigma^2$ is the variance of~$\mu$, $(B_t)_{t\in\r_+}$ is a one-dimensional standard Brownian motion, and $\bar{\nu}_0$ is a suitable probability measure on~$\r$.

To prove our results, we need to introduce the following assumptions.

\begin{assu}
\label{hyp1f}
$\sqrt{f}$ is a positive and Lipschitz continuous function, having Lipschitz constant $L$.
\end{assu}

Under Assumption~\ref{hyp1f}, it is classical that the SDE~\eqref{sde} admits a  unique non-exploding strong solution (see remark~IV.2.1, Theorems~IV.2.3, IV.2.4 and~IV.3.1 of~\cite{ikeda}).

Assumption~\ref{hyp1f} is used in many computations of the paper in one of the following forms:

$\bullet $  $\forall x\in\r,f(x)\leq (\sqrt{f(0)}+L|x|)^2,$\\
or, if we do not need the accurate dependency on the parameter,

$\bullet $  $\forall x\in\r,f(x)\leq C(1+x^2).$

\begin{assu}
\label{hyp1control}
$ $
\begin{itemize}
\item $\int_\r x^4d\bar{\nu}_0(x)<\infty$ and for every $N\in\n^*,$ $\int_\r x^4d\nu^N_0(x)<\infty.$
\item $\mu$ is a centered probability measure having a finite fourth moment, we note $\sigma^2$ its variance.
\end{itemize}
\end{assu}

Assumption~\ref{hyp1control} allows us to control the moments up to order four of the processes $(X^N_t)_t$ and $(\bar{X}_t)_t$ (see Lemma~\ref{controlmoments}) and to prove the convergence of the generators of the processes $(X_t^N)_t$ (see Proposition~\ref{convergencegenerator}).

\begin{assu}
\label{hyp2}
We assume that $f$ belongs to $C^4(\mathbb{R})$ and that for each $1\leq k\leq 4, (\sqrt{f})^{(k)}$ is bounded by some constant $m_k$.
\end{assu}

\begin{rem}
By definition $m_1=L,$ since $m_1:=||(\sqrt{f})'||_\infty$ and $L$ is the Lipschitz constant of~$\sqrt{f}.$
\end{rem}

Assumption~\ref{hyp2} guarantees that the stochastic flow associated to~\eqref{sde} has regularity properties with respect to the initial condition $\bar{X}_0=x$. This will be the main tool to obtain uniform, in time, estimates of the limit semigroup, see Proposition~\ref{regularityPt}.

\begin{example}
The functions $f(x)=1+x^2, f(x)=\sqrt{1+x^2}$ and $f(x)=(\pi/2+\arctan x)^2$ satisfy Assumptions~\ref{hyp1f} and~\ref{hyp2}.
\end{example}

\begin{assu}
\label{hyp1w}
$X^N_0$ converges in distribution to $\bar{X}_0$.
\end{assu}

Obviously, Assumption~\ref{hyp1w} is a necessary condition for the convergence in distribution of $X^N$ to~$\bar{X}.$

\subsection{Main results}\label{main}

Our first main result is the convergence of the process $X^N$ to $\bar{X}$ in distribution in the Skorokhod space, with an explicit rate of convergence  for their semigroups. This rate of convergence will be expressed in terms of the following parameters 
\begin{equation}\label{eq:beta}
\beta : = \max\left(\frac12\sigma^2L^2-\alpha,2\sigma^2L^2-2\alpha,\frac72\sigma^2L^2-3\alpha\right)
\end{equation}
and, for any $ T > 0 $ and any fixed $ \varepsilon > 0, $   
\begin{equation}\label{eq:kt}
K_T :=  \left(1+1/\eps\right)\int_0^T(1+s^2)e^{\beta s}\left(1+e^{(\sigma^2L^2-2\alpha+\eps)(T-s)}\right)ds .
\end{equation}

\begin{rem}
%In the notation $K_T$, we do not emphasize the dependence in $\eps$. The role of $\eps$ is the following: if the condition $\alpha>\sigma^2L^2/2$ holds true, we can fix some $\eps>0$ such that $\sigma^2L^2-2\alpha+\eps<0$, and if no condition on $\alpha$ is required, we can fix any value of $\eps>0.$ In every case, $\eps$ is considered to be fixed.
If $\alpha>7/6\,\sigma^2L^2,$ then $\beta<0,$ and one can choose $\eps>0$ such that $\sigma^2L^2-2\alpha+\eps<0,$ implying that $\sup_{T>0}K_T<\infty.$
\end{rem}

Recall that $\bar{P}_tg(x)=\espcc{x}{g(\bar{X}_t)}$ and $P_t^Ng(x)=\espccc{x}{N}{g(X_t^N)}$ denote the Markovian semigroups of $\bar{X}$ and $X^N$.

\begin{thm}
\label{convergencesemigroupe1}
If Assumptions~\ref{hyp1f} and~\ref{hyp1control} hold, then the following assertions are true.
\begin{itemize}
\item[(i)] Under Assumption~\ref{hyp2}, for all $T\geq 0,$ for each $g\in C_b^3(\r)$ and $x\in\r,$
$$\underset{0\leq t\leq T}{\sup}~\left|P_t^Ng(x)-\bar{P}_tg(x)\right|\leq C (1+x^2)K_T||g||_{3,\infty}\frac{1}{\sqrt{N}}.$$
In particular, if $\alpha>\frac76\sigma^2L^2,$ then 
$$ \underset{ t \geq 0 }{\sup}~\left|P_t^Ng(x)-\bar{P}_tg(x)\right|\leq C (1+x^2) ||g||_{3,\infty}\frac{1}{\sqrt{N}}.$$
\item[(ii)] If in addition Assumption~\ref{hyp1w} holds, then $(X^N)_N$ converges in distribution to~$\bar{X}$ in $D(\r_+,\r)$.
\end{itemize}
\end{thm}

We refer to Proposition~\ref{regularityPt} for the form of $\beta$ given in \eqref{eq:beta}. Theorem~\ref{convergencesemigroupe1} is proved in the end of Subsection~\ref{convergencesemigroups2}. (ii) is a consequence of Theorem~IX.4.21 of \cite{jacod}, using that $X^N$ is a semimartingale. Alternatively, it can be proved as a consequence of  (i), using that $X^N$ is a Markov process.

Below we give some simulations of the trajectories of the process $(X^N_t)_{t\geq 0}$ in Figure~\ref{simulation}.

\begin{figure}[!h]
\hspace*{-1cm}\includegraphics[trim = 1.9cm 17cm 0cm 4cm,clip]{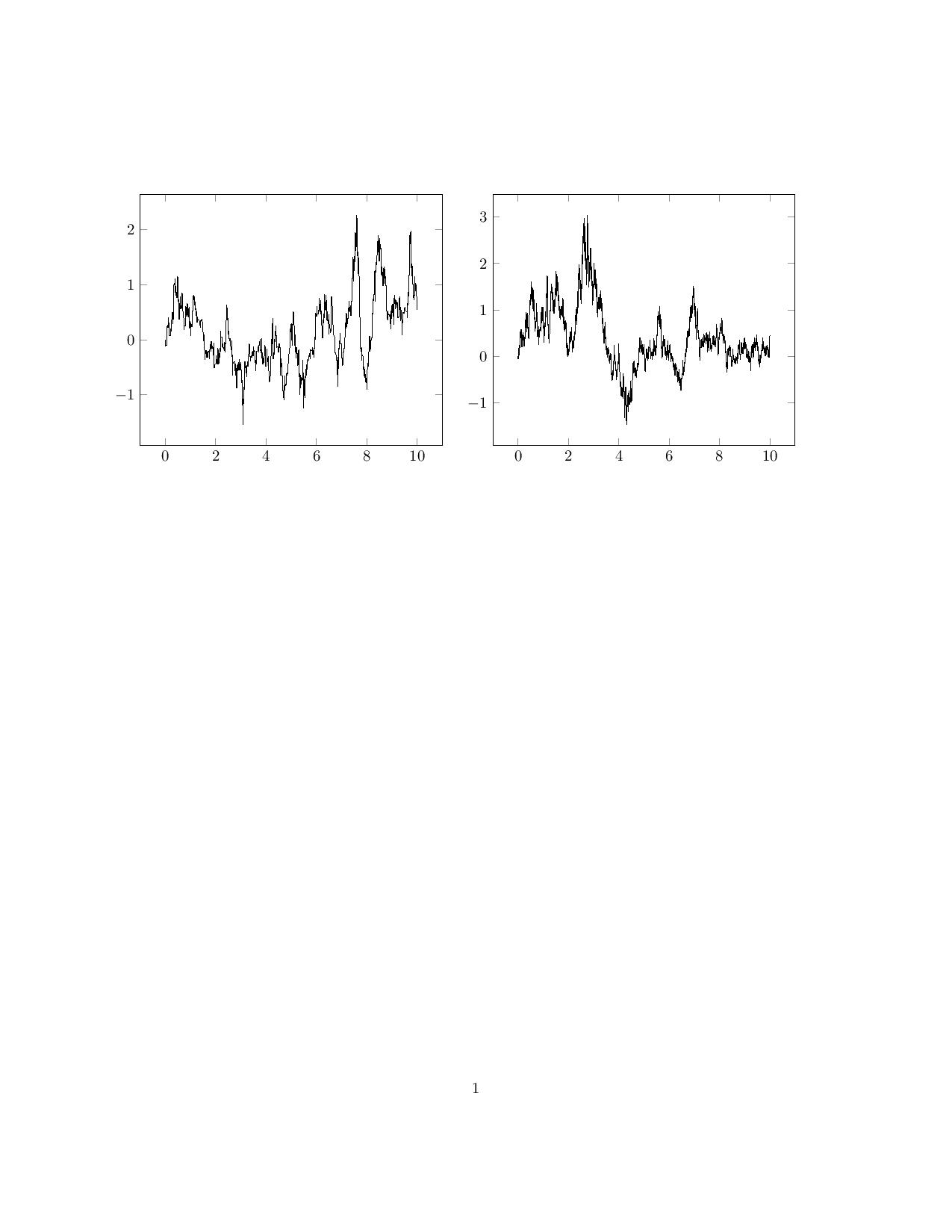}
\caption{Simulation of trajectories of $(X_t^N)_{0\leq t\leq 10}$ with $X^N_0=0$, $\alpha=1$, $\mu=\mathcal{N}(0,1)$, $f(x)=1+x^2,$ $N=100$ (left picture) and $N=500$ (right picture).}
\label{simulation}
\end{figure}

%\begin{figure}[!h]
%\hspace*{-0.4cm}\includegraphics[trim = 1.9cm 17cm 0cm 4cm,clip]{../../../graphique_tex/simulationXtN/graphiqueXtN1050.jpg}
%\caption{Simulation of trajectories of $(X_t^N)_{0\leq t\leq 10}$ with $X^N_0=0$, $\alpha=1$, $\mu=\mathcal{N}(0,1)$, $f(x)=1+x^2,$ $N=10$ (left picture) and $N=50$ (right picture).}
%\label{simulation}
%\end{figure}

\begin{rem}
Theorem~\ref{convergencesemigroupe1}$.(ii)$ states the convergence of $X^N$ to $\bar X$ in the Skorokhod topology. Since $\bar X$ is almost surely continuous, this implies the, a priori stronger, convergence in distribution in the topology of the uniform convergence on compact sets. Indeed, according to Skorohod's representation theorem (see Theorem~6.7 of \cite{billingsley}), we can assume that $X^N$ converges almost surely to $\bar X$ in the Skorokhod space, and this classically entails the uniform convergence on every compact set (see the discussion at the bottom of page~124 in Section~12 of~\cite{billingsley}).
\end{rem}

Under our assumptions, $\bar{P}$ admits an invariant probability measure $\mesinv$, and we can even control the speed of convergence of $P_t^{N}g(x)$ to $\lambda (g)$, as $(N,t)$ goes to infinity, for suitable conditions on the joint convergence of $N$ and $t$.

\begin{thm}
\label{convergencesemigroupe2}
Under Assumptions~\ref{hyp1f} and~\ref{hyp1control}, $\bar{X}$ is recurrent in the sense of Harris, having invariant probability measure $\mesinv (dx) = p(x) dx$ with density 
$$ p (x) = C \frac{1}{f(x) } \exp \left( - \frac{2 \alpha }{\sigma^2 } \int_0^x \frac{y}{f(y)} dy \right) .$$ 
Besides, if Assumption~\ref{hyp2} holds, then for all $g\in C_b^3(\r)$ and $x\in\r,$
$$\left|P_t^N g(x)-\mesinv (g)\right|\leq\constantetruc||g||_{3,\infty}(1+x^2)\left(\frac{K_t}{\sqrt{N}}+e^{-\gamma t}\right),$$
where $\constantetruc$ and $\gamma$ are positive constants independent of $N$ and $t$, and where $K_t$ is defined in \eqref{eq:kt}. In particular, $P_t^N(x,\cdot)$ converges weakly to $\mesinv$ as $(N,t)\rightarrow(\infty,\infty)$, provided $K_t = o(\sqrt{N}).$

If we assume, in addition, that $\alpha>\frac76\sigma^2L^2$, then $P_t^N(x,\cdot)$ converges weakly to $\mesinv$ as $(N,t)\rightarrow(\infty,\infty)$ without any condition on the joint convergence of $(t,N)$, and we have, for any $g\in C_b^3(\r)$ and $x\in\r,$
$$\left|P_t^N g(x)-\mesinv (g)\right|\leq\constantetruc||g||_{3,\infty}(1+x^2)\left(\frac{1}{\sqrt{N}}+e^{-\gamma t}\right).$$
\end{thm}

Theorem~\ref{convergencesemigroupe2} is proved in the end of Section~\ref{convergencetransition}.

Finally, using Theorem~\ref{convergencesemigroupe1}$.(ii)$, we show the convergence of the point processes $Z^{N,i}$ defined in~\eqref{ZNit} to limit point processes $\bar{Z}^i$ having stochastic intensity $ f(\bar X_t) $ at time $t.$  To define the processes $\bar{Z}^i$ ($i\in\n^*$), we fix a Brownian motion $(B_t)_{t\geq 0}$ on some probability space different from the one where the processes $X^N$ ($N\in\n^*$) and the Poisson random measures $\pi_i$ ($i\in\n^*$) are defined. Then we fix a family of i.i.d. Poisson random measures $\bar{\pi}_i$ ($i\in\n^*$) on the same space as $(B_t)_{t\geq 0} , $  independent of $(B_t)_{t\geq 0}$. The limit point processes $ \bar Z^i $ are then defined by
\begin{equation}
\label{barZit}
\bar{Z}^i_t=\int_{]0,t]\times\r_+\times\r}\uno{z\leq f\left(\bar{X}_{s}\right)}d\bar\pi_i(s,z,u).
\end{equation}

\begin{thm}
\label{convergenceZNZ}
Under Assumptions~\ref{hyp1f},~\ref{hyp1control} and~\ref{hyp1w}, for every $k\in\n^*,$ the sequence $(Z^{N,1},\hdots,Z^{N,k})_N$ converges to $(\bar{Z}^1,\hdots,\bar{Z}^k)$ in distribution in $D(\r_+,\r^k)$. Consequently, the sequence $(Z^{N,j})_{j\geq 1}$ converges to $(\bar Z^j)_{j\geq 1}$ in distribution in $D(\r_+,\r)^{\n^*}$ for the product topology.
\end{thm}
Let us give a brief interpretation of the above result. Conditionally on $\bar{X}$, for any $ k > 1, $  $\bar{Z}^1, \ldots, \bar{Z}^k$ are independent. Therefore, the above result can be interpreted as a conditional propagation of chaos property (compare  to \cite{delarue} dealing with the situation where all interacting components are subject to common noise). In our case, the common noise, that is, the Brownian motion $ B$ driving the dynamic of $ \bar X,$ emerges in the limit as a consequence of the central limit theorem. 
Theorem~\ref{convergenceZNZ} is proved in the end of Section~\ref{convergencepoint}.

\begin{rem}
In Theorem~\ref{convergenceZNZ}, we implicitly define $Z^{N,i}:=0$ for each $i\geq N+1$.
\end{rem}

%\begin{rem}
%However, in Theorem~\ref{convergenceZNZ} the convergence in the space $D(\r_+,\r^{\n^*})$ cannot hold true. Indeed, for every $N\in\n^*,$ consider any $k\geq N+1,$ then $Z^{N,k}:=0$ and $\bar Z^k$ is a point process having non-zero intensity, TODOTODO
%\end{rem}

\section{Proof of Theorem~\ref{convergencesemigroupe1}}\label{convergencesemigroups}

The goal of this section is to prove Theorem~\ref{convergencesemigroupe1}. To prove the convergence of the semigroups of $(X^N)_N$, we show in a first time  the convergence of their generators. We start with useful a priori bounds on the moments of $X^N$ and $\bar{X}$.

\begin{lem}
\label{controlmoments}
Under Assumptions~\ref{hyp1f} and~\ref{hyp1control}, the following holds.
\begin{itemize}
\item[(i)] For all $\eps>0,~t>0$ and $x\in\r,~\espcc{x}{(X^N_t)^2}\leq \constantetruc(1+1/\eps)(1+x^2)(1+e^{(\sigma^2L^2-2\alpha+\eps)t}),$ for some $\constantetruc>0$ independent of $N, t,x$ and $\eps.$
\item[(ii)] For all $\eps>0,~t>0$ and $x\in\r,~\espcc{x}{(\bar{X}_t)^2}\leq \constantetruc(1+1/\eps)(1+x^2)(1+e^{(\sigma^2L^2-2\alpha+\eps)t}),$ for some $\constantetruc>0$ independent of $t,x$ and $\eps.$
\item[(iii)] For all $N\in\n^*,T>0,~\esp{({\sup}_{0\leq t\leq T}~|X_t^N|)^2}<+\infty$ and $\esp{({\sup}_{0\leq t\leq T}~|\bar{X}_t|)^2}<+\infty.$
\item[(iv)] For all $T>0,N\in\n^*,~\underset{0\leq t\leq T}{\sup}\espcc{x}{(X_t^N)^4}\leq \constantetruc_T(1+x^4)$ and $\underset{0\leq t\leq T}{\sup}\espcc{x}{(\bar{X}_t)^4}\leq \constantetruc_T(1+x^4).$
\item [(v)] {For all $0\leq s,t\leq T$ and $x\in\r$, $$\espcc{x}{\left(\bar{X}_{t}-\bar{X}_s\right)^2}\leq \constantetruc_{T}(1+x^2)|t-s|\textrm{ and }\espccc{x}{N}{\left(X_{t}^N-X_s^N\right)^2}\leq\constantetruc_{T}(1+x^2)|t-s|.$$}

\end{itemize}
\end{lem}

We postpone the proof of Lemma~\ref{controlmoments} to the Appendix. The inequalities of points $(i)$ and $(ii)$ of the lemma hold for any fixed $\eps>0$. This parameter $\eps$ appears for the following reason. We prove the above points using the Lyapunov function $x\mapsto x^2.$ When applying the generators to this function, there are terms of order $x$ that appear and that we bound by $x^2\eps+\eps^{-1}$ to be able to compare it to~$x^2.$

%Roughly speaking, we prove these points showing a Lyapunov condition on the infinitesimal generators with the function $x\mapsto x^2,$ and in the bounds of the generators, there are terms $x$ that we bound by $x^2\eps+\eps^{-1}.$

\subsection{Convergence of the generators}\label{convergencegenerators}

Throughout this paper, we consider extended generators similar to those used in~\cite{meyn} and in~\cite{davis}, because the classical notion of generator does not suit to our framework (see the beginning of Section~\ref{extendedgenerator}). As this definition slightly differs from one reference to another, we define explicitly the extended generator in Definition~\ref{definitiongenerateur} below and we prove the results on extended generators that we need in this paper. We note $A^N$ the extended generator of $X^N$ and $\bar{A}$ the one of $\bar{X}$, and $\mathcal{D}'(A^N)$ and $\mathcal{D}'(\bar{A})$ their extended domains. The goal of this section is to prove the convergence of $A^Ng(x)$ to $\bar{A}g(x)$ and to establish the rate of convergence for test functions $g\in C_b^3(\r)$. Before proving this convergence, we state a lemma which characterizes the generators for some test functions. This lemma is a straightforward consequence of It\^o's formula and Lemma~\ref{controlmoments}.(i).

\begin{lem}
\label{expressiongenerateur}
$C_b^2(\r)\subseteq\mathcal{D}'(\bar{A})$, and for all $g\in C_b^2(\r)$ and $x\in\r$, we have
$$\bar{A}g(x)=-\alpha xg'(x)+\frac{1}{2}\sigma^2f(x)g''(x).$$
Moreover, $C_b^1(\r)\subseteq\mathcal{D}'(A^N)$, and for all $g\in C_b^1(\r)$ and $x\in\r$, we have
$$A^Ng(x)=-\alpha x g'(x)+Nf(x)\int_\r\left[g\left(x+\frac{u}{\sqrt{N}}\right)-g(x)\right]d\mu(u).$$
\end{lem}
The following result is the first step towards the proof of our main result.
\begin{prop}
\label{convergencegenerator}
If Assumptions~\ref{hyp1f} and~\ref{hyp1control} hold, then for all $g\in C^3_b(\r)$,
$$\left|\bar{A}g(x)-A^Ng(x)\right|\leq f(x)\, \|g'''\|_\infty\frac{1}{6\sqrt{N}}\int_{\r}|u|^3d\mu(u).$$
\end{prop}

\begin{proof}
For $g\in C_b^3(\r)$, if we note $U$ a random variable having distribution~$\mu$, we have, since $ \esp{U} = 0, $ 
\begin{align*}
\left|A^Ng(x)-\bar{A}g(x)\right|\leq&f(x)\left|N\esp{g\left(x+\frac{U}{\sqrt{N}}\right)-g(x)}-\frac{1}{2}\sigma^2g''(x)\right|\\
=&f(x)N\left|\esp{g\left(x+\frac{U}{\sqrt{N}}\right)-g(x)-\frac{U}{\sqrt{N}}g'(x)-\frac{U^2}{2N}g''(x)}\right|\vspace{0.1cm}\\
\leq&f(x)N\esp{\left|g\left(x+\frac{U}{\sqrt{N}}\right)-g(x)-\frac{U}{\sqrt{N}}g'(x)-\frac{U^2}{2N}g''(x)\right|}.
\end{align*}

Using Taylor-Lagrange's inequality, we obtain the result.
\end{proof}

\subsection{Convergence of the semigroups}\label{convergencesemigroups2}

Once the convergence $A^Ng(x)\rightarrow \bar{A}g(x)$ is established, together with a control of the speed of convergence, our strategy is to rely on the following representation 
\begin{equation}\label{TBbis}
\left(\bar{P}_t-P^N_t\right)g(x)=\int_0^tP^N_{t-s}\left(\bar{A}-A^N\right)\bar{P}_sg(x)ds,
\end{equation}
which is proven in Proposition~\ref{generatorsemigroup} in the Appendix.

Obviously, to be able to apply Proposition \ref{convergencegenerator} to the above formula, we need to ensure the regularity of $x\mapsto \bar{P}_sg(x),$ together with a control of the associated norm $||\bar{P}_sg||_{3, \infty}$. This is done in the next proposition.

\begin{prop}
\label{regularityPt}
If Assumptions~\ref{hyp1f},~\ref{hyp1control} and~\ref{hyp2} hold, then for all $t\geq 0$ and for all $g\in C^3_b(\r)$, the function $x\mapsto\bar{P}_tg (x) $ belongs to $C^3_b(\r)$ and satisfies
\begin{equation}\label{eq:qt1}
\left|\left|\left(\bar{P}_tg\right)'''\right|\right|_\infty\leq \constantetruc||g||_{3,\infty}(1+t^2)e^{\beta t},
\end{equation}
with $\beta = \max(\frac12\sigma^2L^2-\alpha,2\sigma^2L^2-2\alpha,\frac72\sigma^2L^2-3\alpha).$
{Moreover, for all $T>0,$ 
\begin{equation}\label{eq:qt2}
{\sup}_{0\leq t\leq T}||\bar{P}_tg||_{3,\infty}\leq Q_T||g||_{3,\infty}
\end{equation}
for some $Q_T>0,$ and
for all  $i\in\{0,1,2\}$ and $x\in\r$, $s\mapsto (\bar{P}_sg)^{(i)}(x)=\frac{\partial^i}{\partial x^i}(\bar{P}_sg(x))$ is continuous.}
\end{prop}

The proof of Proposition~\ref{regularityPt} requires some detailed calculus to obtain the explicit expression for~$\beta$, so we postpone it to the Appendix.

\begin{proof}[Proof of Theorem~\ref{convergencesemigroupe1}]

{\bf Step 1.} The proof of point $(i)$ is a straightforward consequence of Proposition~\ref{convergencegenerator}, since, applying formula \eqref{TBbis}, 
\begin{align*}
\left|\bar{P}_tg(x)-P^N_tg(x)\right|= & \left|\displaystyle\int_0^tP^N_{t-s}\left(\bar{A}-A^N\right)\bar{P}_sg(x)ds\right|\\
\leq&\displaystyle\int_0^t\espccc{x}{N}{\left|\bar{A}\left(\bar{P}_sg\right)\left(X_{t-s}^N\right)-A^N\left(\bar{P}_sg\right)\left(X_{t-s}^N\right)\right|}ds\\
\leq& \constantetruc\frac{1}{\sqrt{N}}\int_0^t\left|\left|\left(\bar{P}_sg\right)'''\right|\right|_\infty\espccc{x}{N}{f\left(X_{t-s}^N\right)}ds\\
\leq& \constantetruc\frac{1}{\sqrt{N}}||g||_{3,\infty}\int_0^t\left((1+s^2)e^{\beta s}\left(1+\espccc{x}{N}{\left(X^N_{t-s}\right)^2}\right)\right)ds\\
\leq& \constantetruc\left(1+\frac1\eps\right)\frac{1}{\sqrt{N}}||g||_{3,\infty}(1+x^2)\int_0^t(1+s^2)e^{\beta s}\left(1+e^{(\sigma^2L^2-2\alpha+\eps)(t-s)}\right)ds,
\end{align*}
where we have used respectively Proposition~\ref{regularityPt} and Lemma~\ref{controlmoments}.(i) to obtain the two last inequalities above, and $\eps$ is any positive constant.

{\bf Step 2.} We now give the proof of point $(ii)$ of the theorem. With the notation of Theorem~$IX.4.21$ of \cite{jacod}, we have $K^N(x,dy):=Nf(x)\mu(\sqrt{N}dy),$ $b'^N(x)=-\alpha x+\int K^N(x,dy)y=-\alpha x$, and $c'^N(x)=\int K^N(x,dy)y^2=\sigma^2f(x).$ Then, an immediate adaptation of Theorem~$IX.4.21$ of \cite{jacod} to our frame implies the result.
\end{proof}

%
%We finally state the 
%\begin{lem}
%\label{convergencetechnique}
%Let $(g_k)_k$ be a sequence of $C_b^1(\r)$ satisfying ${\sup}_k||g_k'||_{\infty}<\infty$, and for all $x\in\r,~g_k(x)\rightarrow 0$ as $k\rightarrow\infty$.
%
%Then for all bounded sequences of real numbers $(x_k)_k,~g_k(x_k)\rightarrow 0$ as $k\rightarrow\infty$.
%\end{lem}
%%The proof of this result is straightforward. 
%\begin{proof} It follows from ${\sup}_k||g_k'||_{\infty}<\infty$ that the sequence $(g_k)$ is uniformly equicontinuous, hence its point-wise convergence to zero implies the uniform convergence to zero on each compact.\end{proof}
%Let $(x_k)_k$ be a bounded sequence. In a first time, we suppose that $(x_k)_k$ converges to some $x\in\r$. Then we have $|g_k(x_k)|\leq ||g_k'||_\infty|x-x_k|+|g_k(x)|$ which converges to zero as $k$ goes to infinity. In the general case, we show that for all subsequence of $(g_k(x_k))_k$, there exists a subsequence of the first one that converges to~0 (the second subsequence has to be chosen such that $x_k$ converges).
%\end{proof}

\section{Proof of Theorem~\ref{convergencesemigroupe2}}\label{convergencetransition}

In this section, we prove Theorem~\ref{convergencesemigroupe2}. We begin by proving some properties of the invariant measure of $\bar{P}_t. $ In what follows we use the total variation distance between two probability measures $ \nu_1 $ and $ \nu_2$ defined by 
$$ \|\nu_1 - \nu_2 \|_{TV} = \frac12  \sup _{ g : \| g \|_\infty \le 1 } | \nu_1 (g) - \nu_2 (g) | .$$
\begin{prop}
\label{ergodic}
If Assumptions~\ref{hyp1f} and~\ref{hyp1control} hold, then the invariant measure $\mesinv$ of $(\bar{P}_t)_t$ exists and is unique. Its density is given, up to multiplication with a constant, by 
$$ p (x) = C \frac{1}{f(x) } \exp \left(-  \frac{2 \alpha }{\sigma^2 } \int_0^x \frac{y}{f(y)} dy \right) .$$
In addition, if Assumption~\ref{hyp2} holds, then for every $0<q<1/2,$ there exists some $\gamma>0$ such that, for all $t\geq 0,$
$$||\bar{P}_t(x,\cdot)-\mesinv||_{TV}\leq \constantetruc \left(1+x^2\right)^qe^{-\gamma t}.$$
\end{prop}

\begin{proof}
In a first time, let us prove the positive Harris recurrence of $\bar{X}$ implying the existence and uniqueness of $\mesinv$. According to Example 3.10 of \cite{khasminskii_stochastic_2012} it is sufficient to show that $S(x):=\int_0^xs(y)dy$ goes to $+\infty$ (resp. $-\infty$) as $x$ goes to $+\infty$ (resp. $-\infty$), where
$$s(x):=\exp\left(\frac{2\alpha}{\sigma^2}\int_0^x\frac{v}{f(v)}dv\right).$$
For $x>0,$ and using that $f$ is subquadratic, 
\begin{equation*}
s(x)\geq\exp\left(\constantetruc\int_0^x\frac{2v}{1+v^2}dv\right)=\exp\left(\constantetruc\ln(1+x^2)\right) = (1+x^2)^\constantetruc\geq 1,
\end{equation*}
implying that $S(x)$ goes to $+\infty$ as $x$ goes to $+\infty.$ With the same reasoning, we obtain that $S(x)$ goes to $-\infty$ as $x$ goes to $-\infty.$ Finally, the associated invariant density is given, up to a constant, by 
$$ p(x) = \frac{C}{f(x) s(x) } .$$ 
For the second part of the proof, take $ V(x) = (1 + x^2)^q , $ for some $ q < 1/2,$ then 
$$ V'(x) = 2q x (1+x^2)^{q-1} , V'' (x) = 2 q  (1 + x^2)^{q-2} [ 2 x^2 (q-1) + (1+x^2) ] .$$
As $q < \frac12,$ $V'' (x) < 0 $ for $x^2 $ sufficiently large, say, for  $ |x|\geq K.$  In this case, for $ |x|\geq K,$
$$ \bar A V (x) \le - 2 \alpha q x^2 ( 1+x^2)^{q-1} \le - 2 \alpha q \frac{x^2}{1+x^2 } V(x) \le - 2 q \alpha \frac{K^2}{1+K^2 } V(x)  = - c V(x)  .$$ 
So we obtain that, for suitable constants $c$ and $d$, for any $x\in\r,$
\begin{equation}
\label{lyapunov}
\bar A V(x) \le - c V(x) + d .
\end{equation}

%Using Proposition~15 of \cite{fort_subgeometric_2005}, all compact sets are petite. Then, Theorem~6.1 of \cite{meyn} implies the following bound: introducing for any probability measure $\mu $ the weighted norm
Obviously, for any fixed $ T > 0, $  the sampled chain $(\bar X_{kT})_{k\geq 0}$ is Feller and $\mesinv-$irreducible. The support of $\mesinv$ being $\r$, Theorem 3.4 of \cite{meyn} implies that every compact set is petite for the sampled chain. Then, as \eqref{lyapunov} implies the condition $(CD3)$ of Theorem~6.1 of \cite{meyn}, we have the following bound: introducing for any probability measure $\mu $ the weighted norm

$$\| \mu\|_V := \sup_{ g : |g| \le 1+V } | \mu (g) |, $$
there exist $C,\gamma>0$ such that
$$ \| \bar P_t (x, \cdot ) - \mesinv \|_V \le C (1+V(x)) e^{- \gamma t }.$$
This implies the result, since $||\cdot||_{TV}\leq ||\cdot||_V$.
\end{proof}

Now the proof of Theorem~\ref{convergencesemigroupe2} is straightforward.

\begin{proof}[Proof of Theorem~\ref{convergencesemigroupe2}]
The first part of the theorem has been proved in Proposition~\ref{ergodic}. For the second part, for any $g\in C_b^3(\r),$
\begin{align*}
\left|P_t^Ng(x)-\mesinv (g)\right|\leq &\left|P_t^Ng(x)-\bar{P}_tg(x)\right|+\left|\bar P_tg(x)-\mesinv (g)\right|\\
\leq& \frac{K_t}{\sqrt{N}}(1+x^2)||g||_{3,\infty}+||g||_\infty||\bar P_t(x,\cdot)-\mesinv||_{TV}\\
\leq& ||g||_{3,\infty}\constantetruc\left(\frac{K_t}{\sqrt{N}}(1+x^2)+e^{-\gamma t}(1+x^2)^q\right),
\end{align*}
where we have used Theorem~\ref{convergencesemigroupe1} and Proposition~\ref{ergodic}. Since $ (1+x^2)^q \le 1 + x^2 ,$ $q$ being smaller than $1/2,$ this implies the result. 
\end{proof}

\section{Proof of Theorem~\ref{convergenceZNZ}}\label{convergencepoint}

We prove the result using Theorem~IX.4.15 of \cite{jacod}.

%Let $k\in\n^*$, let us note $Y^N:=(X^N,Z^{N,1},\hdots,Z^{N,k})$ and $\bar Y:=(\bar X,\bar Z^1,\hdots,\bar Z^k).$ Using the notation of Theorem~IX.4.15 of \cite{jacod} with the semimartingales $Y^N$ ($N\in\n^*$) and $\bar Y$ and denoting $e^j$ ($1\leq j\leq k+1$) the $j-$th unit vector, we have:
%\begin{itemize}
%\item $b'^{N,1} (x) = b'^1 (x)= -\alpha x$ and $b'^{N,i}(x)=b'^i(x)=0$ for $2\leq i\leq k+1,$
%\item $\tilde c'^{N,1,1} (x) = c'^{1,1}(x) = \sigma^2f(x^1)$ and $c'^{N,i,j}(x)=c'^{i,j}(x)=0$ for $(i,j)\neq (1,1),$
%\item $g * K^N (x) = f(x^1)\sum_{j=2}^{k+1} \int_\r g(\frac{u}{\sqrt{N}}e^1+e^{j})d\mu(u)+(N-k)\int_\r g(\frac{u}{\sqrt{N}}e^1)d\mu(u),$
%\item $g * K(x)=f(x^1)\sum_{j=1}^k g(e^{j+1}).$
%\end{itemize}

Let $k\in\n^*$, let us note $Y^N:=(X^N,Z^{N,1},\hdots,Z^{N,k})$ and $\bar Y:=(\bar X,\bar Z^1,\hdots,\bar Z^k).$ Using the notation of Theorem~IX.4.15 of \cite{jacod} with the semimartingales $Y^N$ ($N\in\n^*$) and $\bar Y$ and denoting $e^j$ ($0\leq j\leq k$) the $j-$th unit vector, we have:
\begin{itemize}
\item $b'^{N,0} (x) = b'^0 (x)= -\alpha x$ and $b'^{N,i}(x)=b'^i(x)=0$ for $1\leq i\leq k,$
\item $\tilde c'^{N,0,0} (x) = c'^{0,0}(x) = \sigma^2f(x^0)$ and $c'^{N,i,j}(x)=c'^{i,j}(x)=0$ for $(i,j)\neq (0,0),$
\item $g * K^N (x) = f(x^0)\sum_{j=1}^{k} \int_\r g(\frac{u}{\sqrt{N}}e^0+e^{j})d\mu(u)+(N-k)\int_\r g(\frac{u}{\sqrt{N}}e^0)d\mu(u),$
\item $g * K(x)=f(x^0)\sum_{j=1}^k g(e^{j}).$
\end{itemize}

The only condition of Theorem~IX.4.15 that is not straightforward is the convergence of $g*K^N$ to $g*K$ for $g\in C_1(\r^{k+1}).$ The complete definition of $C_1(\r^{k+1})$ is given in VII.2.7 of \cite{jacod}, but here, we just use the fact that $C_1(\r^{k+1})$ is a subspace of $C_b(\r^{k+1})$ containing functions which are zero around zero. This convergence follows from the fact that any $g\in C_1(\r^{k+1})$ can be written as $g(x)=h(x)\uno{|x|>\eps}$ where $h\in C_b(\r^{k+1})$ and $\eps>0.$ This allows to show that, for this kind of function $g,$
%\begin{multline*}
%\left|(N-k)f(x^1)\int_\r g\left(\frac{u}{\sqrt{N}}e^1\right)d\mu(u)\right|\leq (N-k)f(x^1)||h||_\infty\int_\r\uno{|u|>\eps\sqrt{N}}d\mu(u)\\
%\leq f(x^1)C\frac{N-K}{N^2}\leq Cf(x^1)N^{-1},
%\end{multline*}
\begin{multline*}
\left|(N-k)f(x^0)\int_\r g\left(\frac{u}{\sqrt{N}}e^0\right)d\mu(u)\right|\leq (N-k)f(x^0)||h||_\infty\int_\r\uno{|u|>\eps\sqrt{N}}d\mu(u)\\
\leq f(x^0)C\frac{N-K}{N^2}\leq Cf(x^0)N^{-1},
\end{multline*}
where the second inequality follows from the fact that we assume that $\mu$ is a probability measure having {a finite} fourth moment.

Theorem~IX.4.15 of \cite{jacod} implies that for all $k\geq 1$, $(Z^{N,1},...,Z^{N,k})$ converges to $(\bar Z^1,...,\bar Z^k)$ in distribution in $D(\r_+,\r^k)$.

This implies the weaker convergence in $D(\r_+,\r)^k$ for any $k\in\mathbb{N}^*.$ Then, the convergence in $D(\r_+,\r)^{\mathbb{N}^*}$ is classical (see e.g. Theorem~3.29 of \cite{kallenberg_foundations_1997}).
%This implies the convergence in $D(\r_+,\r)^k.$ In particular, for every $i\geq 1,$ the sequence $(Z^{N,i})_N$ is tight on $D(\r_+,\r)$, whence the sequence $((Z^{N,i})_{i\geq 1})_N$ is tight on $D(\r_+,\r)^{\n^*}$. Indeed let $d_S$ be the Skorokhod metric and $d(x,y):=\sum_{i=1}^{\infty}2^{-i}d_S(x_i,y_i)$. It is classical that the topology of $d$ is the product topology on $D(\r_+,\r)^{\n^*}$. By hypothesis, for every $i\geq 1,\eps>0,$ there exists a compact set $K^{i}_\eps$ such that for all $N,$ $\pro{Z^{N,i}\not\in K^i_\eps}\leq \eps$. Then, by Tykhonoff's theorem, $K_\eps:=\prod_{i=1}^{+\infty}K^i_{\eps/2^i}$ is a compact set of $D(\r_+,\r)^{\n^*}$ and it satisfies, for all $N,\pro{(Z^{N,i})_{i\geq 1}\not\in K_\eps}\leq \eps.$

%So the sequence $((Z^{N,i})_{i\geq 1})_N$ is tight on $D(\r_+,\r)^{\n^*}.$ Then, the projections of any limit in distribution of this sequence on the $k$ first coordinates have the same distribution as $(\bar Z^1,...,\bar Z^k)$, and this implies that the law $(Z^i)_{i\geq 1}$ is the only limit in distribution of $(Z^{N,i})_{i\geq 1}$.

%and, by standard arguments, the convergence of $(Z^{N,i})_{i\geq 1}$ to $(\bar Z^i)_{i\geq 1}$ in $D(\r_+,\r)^{\n^*}.$

\section{Appendix}\label{}

\subsection{Extended generators}\label{extendedgenerator}

%In this subsection, we define precisely the notion of generators we use and we prove the results needed to prove formula~\eqref{formulegenerateur}.

There are different definitions of infinitesimal generators in the literature. The aim of this subsection is to define precisely the notion of generator we use in this paper. Moreover we establish some properties of these generators and prove formula~\eqref{TBbis}. In the general theory of semigroups, one defines the generators on some Banach space. In the frame of semigroups related to Markov processes, one generally considers $(C_b(\r),||\cdot||_\infty)$. In this context, the generator $A$ of a semigroup $(P_t)_t$ is defined on the set of functions $\mathcal{D}(A)=\{g\in C_b(\r):\exists h\in C_b(\r),||\frac{1}{t}(P_tg-g)-h||_\infty{\longrightarrow}\,0\textrm{ as }t\rightarrow 0\}$. Then one denotes the previous function $h$ as $Ag$. In general, we can only guarantee that $\mathcal{D}(A)$ contains the functions that have a compact support, but to prove Proposition~\ref{generatorsemigroup}, we need to apply the generators of the processes $(X_t^N)_t$ and $(\bar{X}_t)_t$ to functions of the type $\bar{P}_sg$, and we cannot guarantee that $\bar{P}_sg$ has compact support even if we assume $g$ to be in $C_c^\infty(\r)$.

This is why we consider extended generators (see for instance~\cite{meyn} or~\cite{davis}). These extended generators are  defined by the point-wise convergence on~$\r$ instead of the uniform convergence. Moreover, they verify the fundamental martingale property, which allows us to define the generator on $C_b^n(\r)$ for suitable $n\in\n^*$ and to prove that some properties of the classical theory of semigroups still hold for this larger class of functions.

Let $(X_t)_t$ be a Markov process taking values in~$\r$. 
We set
\[
\D(P)=
\{g:\R\to\R,\; \mbox{measurable, s.t.}\;  \forall x\in\r,\; \forall t\geq 0,\;\;  \E_x|g(X_t)|<\infty\}.
\]
For  $g\in\D(P)$, $x\in\R,$ $t\geq 0,$  we define $P_tg(x)=\espcc{x}{g(X_t)}.$ 

{
\begin{defi}
\label{definitiongenerateur}
We define $\mathcal{D}'(A)$ to be the set of $g\in\D(P)$ for which there exists a measurable function $Ag:\R\to\R,$ 
 such that $Ag\in\D(P),$ $ t \mapsto P_t Ag (x) $ is continuous in $0,$ and $\forall x\in\R,$ $\forall t\geq 0,$
\begin{description}
\item(i) $\E_x{g(X_t)}-g(x)=\E_x{\int_0^t Ag(X_s)ds};$
\item (ii) $\E_x{\int_0^t|A(g(X_s))|ds}<\infty.$
\end{description}
\end{defi}

\begin{rem}
Using Fubini's theorem and $(ii)$ we can rewrite $(i)$ in the following form:
\begin{equation}\label{fondamental}
P_tg(x)-g(x)=\int_0^tP_sAg(x)ds.
\end{equation}
Then \eqref {fondamental} implies immediately that if $g\in\D'(A),$ then
\begin{equation}
\label{derivezero}
\lim_{t\to 0}\frac{1}{t}(P_tg(x)-g(x))=Ag(x).
\end{equation}
Note also that it follows  from the Markov property and the definition of $ Ag$  that the process 
$ g(X_t)-g(X_0)-\int_0^t Ag(X_s) ds$
 is a $\mathbb{P}_x$-martingale w.r.t. to the filtration generated by $(X_t)_t$.
\end{rem}

The following result is classical and stated without proof. It is a straightforward consequence of \eqref{fondamental} and \eqref{derivezero}. 
\begin{prop}\label{deriveegenerateur}
Suppose that $A$ is the extended generator of the semigroup $(P_t)_t,$ $g\in\D'(A), $ 
%$Ag\in\D(P),$ 
and the map $s\to P_sAg(x)$ is continuous on $\R_+$ for some $x\in\R. $  
Then 
$$\frac d{dt}P_tg(x)=P_tAg(x).$$
Moreover, if for all $t\geq 0,$ $ P_tg\in\D'(A),$ then $\frac d{dt}P_tg(x)=AP_tg(x)=P_tAg(x).$
\end{prop}

%\begin{proof}
%The first assertion follows immediately from \eqref{fondamental} and the continuity on $\R$ of $s\to P_sAg(x)$, hence for $h\geq 0,$
%$$\lim_{h\to 0}\frac{1}{h}(P_{t+h}g(x)-P_tg(x))=\lim_{h\to 0}\frac 1h\int_t^{t+h}P_sAg(x)ds=P_tAg(x),$$
%and
%$$\lim_{h\to 0}\frac{1}{h}(P_{t}g(x)-P_{t-h}g(x))=\lim_{h\to 0}\frac 1h\int_{t-h}^tP_sAg(x)ds=P_tAg(x).$$
%The second assertion is a consequence of  \eqref{derivezero}.
%\end{proof}

In what follows, we give some sufficient conditions to verify the continuity and the derivability  of the map $s\mapsto P_sh(x).$ These conditions are not intended to be optimal, they are stated such that it is easy to check them both for $ X^N $ and $ \bar X . $ 
\begin{prop}\label{practical}
Let $(X_t)_t$ be a Markov process with semigroup $(P_t)_t$ and extended generator $A.$  
\begin{enumerate}
\item Let $h\in \D(P),$ $x\in\R.$ Suppose that
\begin{itemize} 
\item [(i)] the map $t\to X_t$ is continuous  in $\LL^2,$ i.e.
$ \lim_{|t-s|\to 0} \E_x|X_s-X_t|^2= 0 ;$

\item  [(ii)]  for all $T>0$, ${\sup}_{0\leq t\leq T}\E_x{(|X_t|^4)}<+\infty;$
\item [(iii)] there exists $C>0,$ such that $\forall x,  y \in \R,\; $ $|h(x)-h(y)|\leq C(1+x^2+y^2)|x-y|.$\\
\end{itemize}
Then the map $s\mapsto P_sh(x)$ is continuous on $\R_+.$
\item Suppose moreover that $(i)$, $(ii)$ and $(iii)'$ are satisfied with
\begin{itemize}
\item [(iii)'] $g\in \D'(A)$ such that
 for some $C>0,$ 
and for all $x, y\in \R,$ we have that
$|{A}g(x)-{A}g(y)|\leq C(1+x^2+y^2)|x-y|.$
\end{itemize}
Then the map $s\to P_sg(x)$ is differentiable on $\R_+,$ and $\frac d{dt}P_tg(x)=P_tAg(x).$

%$s\to P_sAg(x)$ is continuous on $\R_+$, and as a consequence,  $\frac d{dt}P_tg(x)=P_tAg(x).$

\end{enumerate}

\end{prop}

\begin{proof} 
The proof of  point 1.  follows from the following chain of inequalities
\begin{multline*}
|P_th(x)-P_sh(x)|\leq \E_x |h(X_t)-h(X_s)|\leq C \E_x\left[(1+X_t^2+X_s^2)|X_t-X_s|\right ]\leq 
\\ C [\E_x(1+X_t^4+X_s^4)]^{1/2}[\E_x|X_t-X_s|^2]^{1/2}\leq C \sup_{u\leq s\vee t}[\E_x X_u^4]^{1/2}\|X_s-X_t\|_2\underset{{|t-s|\to 0}}{{\longrightarrow 0}} .
\end{multline*}
The second assertion of the proof follows from point 1. and Proposition \ref{deriveegenerateur}, observing that $ h := Ag$ satisfies point (iii). 

\end{proof}

\subsection{Proof of \eqref{TBbis}}
In this section, we  first collect some useful results about the extended generators $A^N $ of $ X^N $ and $ \bar A $ of $ \bar X. $ Then we give the proof of \eqref{TBbis}. We start with the following result.
\begin{prop}\label{derivegenerateur}
1. For all  $g\in C_b^3(\r)$, for all $x,y\in\r,$
$$|\bar{A}g(x)-\bar{A}g(y)|\leq \constantetruc \|g\|_{3, \infty } (1+x^2+y^2)|x-y|\;   \mbox{ and } |\bar{A}g(x)|\leq \constantetruc||g||_{2, \infty}(1+x^2) .$$
In particular,  for any $g \in C^3_b, $ the map $t \to \bar P_tg(x)$ is differentiable on $\R_+,$ and $\frac d{dt}\bar P_tg(x)=\bar P_t \bar Ag(x) = \bar A \bar P_t g ( x) .$\\
2. For all $ g \in C_b^2 ( \r), $ for all $ x, y \in \r,$  
$$  |A^Ng(x)-A^Ng(y)|\leq \constantetruc \|g\|_{2, \infty } (1+x^2+y^2)|x-y| \mbox{ and } \;  |A^Ng(x)|\leq \constantetruc||g||_{1, \infty}(1+x^2).$$
In particular,  for any $g \in C^2_b, $ the map $ t \to P^N_t g( x) $ is differentiable on $\R_+,$ and $\frac d{dt} P^N_tg(x)= P^N_t  A^Ng(x) .$ 
\end{prop}

\begin{proof}
The result follows from Proposition \ref{practical} together with Lemma \ref{controlmoments} and Lemma \ref{expressiongenerateur}. Finally, to show that $\bar P_t \bar Ag(x) = \bar A \bar P_t g ( x) ,$ we use Proposition \ref{deriveegenerateur} and Proposition \ref{regularityPt}. 
\end{proof}}

We are now able to give the proof of the main result of this section. This result is a Trotter-Kato like formula that allows to obtain a control of the difference between the semigroups $\bar P$ and $P^N$, provided we dispose already of a control of the distance between their generators. It is an adaptation of Lemma~1.6.2 from \cite{ethier} to the notion of extended generators.

\begin{prop}
\label{generatorsemigroup}
Grant Assumptions~\ref{hyp1f},~\ref{hyp1control} and~\ref{hyp2}. Let $\bar A$ and $A^N$ be the extended generators of respectively $\bar P$ and $P^N.$

Then the following equality holds for each $g\in C^3_b(\r)$, $x\in\r$ and $t\in\r_+ .$
\begin{equation}
\label{TB}
\left(\bar{P}_t-P^N_t\right)g(x)=\int_0^tP_{t-s}^N\left(\bar{A}-A^N\right)\bar{P}_sg(x)ds.
\end{equation}
\end{prop}
\begin{proof}

%To begin with, let us emphasize the fact that hypothesis $(i)$ implies
%\begin{equation}
%\label{hyp(i)}
%\underset{0\leq t\leq T}{\sup}\espcc{x}{(\bar{Y}_t)^2}\leq \constantetruc_T(1+x^2)\textrm{ and }\underset{0\leq t\leq T}{\sup}\espccc{x}{N}{(Y^N_t)^2}\leq \constantetruc_T(1+x^2),
%\end{equation}
%since
%$$\espcc{x}{(\bar Y_t)^2}\leq\espcc{x}{(\bar Y_t)^4}^{1/2}\leq \constantetruc_T\sqrt{1+x^4}\leq\constantetruc_T\ll(1+x^2\rr).$$
%
We fix $t\geq 0,N\in\n^*,g\in C^3_b(\r),x\in\r$ in the rest of the proof. Introduce for $ 0 \le s \le t $ the function $u(s)=P_{t-s}^N\bar{P}_sg(x)$.

One can note that $u = \Phi \circ \Psi$ with $\Phi:\R^2\to\R;$ $\Phi(v_1,v_2) = P^N_{v_1}\bar P_{v_2}g(x)$ and $\Psi:\R\to\R^2;$ $\Psi(s) = (t-s,s)$. Let us show that $\Phi$ is differentiable w.r.t. to both variables $v_1$ and $v_2$. Indeed, for $v_1$ it is a consequence of the fact  that $h=\bar P_{v_2}g\in C^3_b(\mathbb{R})$ by Proposition \ref{regularityPt}
 and Proposition~\ref{derivegenerateur}, from which we know  that if $h\in C^2_b,$ then $v_1\to P^N_{v_1}h(x)$ is differentiable and 
 $$\frac {\partial }{dv_1} \Phi(v_1,v_2)=\frac d{dv_1}P^N_{v_1}h(x)=P^N_{v_1}A^Nh(x).$$  
  To show the differentiability of $\Phi$ with respect to  $v_2$, let us write
  $\Phi(v_1,v_2) = \espcc{x}{\bar P_{v_2}g(X^N_{v_1})}.$  From Proposition~\ref{derivegenerateur}, we know that since $g\in C^3_b,$ $v_2\mapsto \bar P_{v_2}g(X^N_{v_1})$ is a.s. differentiable with  derivative 
 $$\frac d{dv_2} \bar P_{v_2}g(X^N_{v_1})=\bar A\bar P_{v_2}g(X^N_{v_1})=\bar P_{v_2}\bar Ag(X^N_{v_1})=\E_{X_{v_1}^N}(\bar A g)(\bar X_{v_2}).$$
 Moreover, 
 $|\bar Ag(x)|\leq \constantetruc||g||_{2,\infty}(1+x^2)$ by Proposition \ref{derivegenerateur}. Now, using Lemma~\ref{controlmoments}.$(ii)$ we see that 
$$\sup_{v_2\leq T} \left|\frac d{dv_2} \bar P_{v_2}g(X^N_{v_1})\right|\leq \E_{X^N_{v_1}}\left[\sup_{v_2\leq T}|(\bar A g)(\bar X_{v_2})|\right]\leq C_T  \left(1+(X^N_{v_1})^2\right).$$ 
By Lemma~\ref{controlmoments}.$(iii),$ we see that the last bound is integrable, hence by dominated convergence, $v_2\mapsto \Phi(v_1,v_2) $ is differentiable with derivative 
$$\frac {\partial }{dv_2} \Phi(v_1,v_2)=P^N_{v_1}\bar A\bar P_{v_2}g(x)=P^N_{v_1}\bar P_{v_2}\bar A g(x).$$  
 As a consequence, $u$ is differentiable on $\R_+ ,$ and we have 
\begin{align*}
u'(s)=&-\frac{\partial}{\partial v_1} \Phi(t-s,s) + \frac {\partial}{\partial v_2} \Phi(t-s,s)\\
=&-P^N_{t-s}A^N\bar{P}_sg(x)+P^N_{t-s}\bar{P}_s\bar{A}g(x)\\
=&P^N_{t-s}\left(\bar{A}-A^N\right)\bar{P}_sg(x).
\end{align*}
Now we show that $u'$ is continuous. Indeed, if it is the case, then we will have
$$u(t)-u(0)=\int_0^tu'(s)ds,$$
which is exactly the assertion. 

In order to prove the continuity of $u'$, we consider a sequence $(s_k)_k$ that converges to some $s\in[0,t]$, and we write
\begin{align}
\left|P_{t-s}^N\left(\bar{A}-A^N\right)\bar{P}_sg(x)-P_{t-s_k}^N\left(\bar{A}-A^N\right)\bar{P}_{s_k}g(x)\right|\leq& \left|\left(P_{t-s}^N-P_{t-s_k}^N\right)\left(\bar{A}-A^N\right)g_s(x)\right|\label{termenumero1}\\&+\left|P_{t-s_k}^N\left(\bar{A}-A^N\right)\left(\bar{P}_s-\bar{P}_{s_k}\right)g(x)\right|,\label{termenumero2}
\end{align}
where $g_s=\bar{P}_sg\in C_b^3(\r)$.

To show that the term~\eqref{termenumero1} vanishes when $k$ goes to infinity, denote $h_s(x)=(\bar A-A^N)g_s(x).$ Using Proposition \ref{derivegenerateur} and the fact that $g_s\in C_b^3(\r)$,
we have
$$|h_s(x)-h_s(y)|\leq C (1+x^2+y^2)|x-y|.$$
Proposition \ref{practical} applied to $h_s $ and to $ P^N$ implies that $u\to P_u^Nh_s(x)$ is continuous. As a consequence the term~\eqref{termenumero1} vanishes as $ k \to \infty .$

To finish the proof, we need to show that the term~\eqref{termenumero2} vanishes.
Denote $g_k=\left(\bar{P}_s-\bar{P}_{s_k}\right)g .$  We have to show that
$$\E_x \left[  \left(\bar{A}-A^N\right)g_k(X^N_{t-s_k})\right] \to 0, \;  \mbox {when}\; k\to \infty.$$
In what follows we will in fact show that 
\begin{equation}\label{eq:exp}
\E_x \left[ \bar{A}g_k(X^N_{t-s_k})\right]  \to 0 \mbox{ and } \E_x \left[ {A^N}g_k(X^N_{t-s_k})\right]  \to 0 ,  \;  \mbox {when}\; k\to \infty.
\end{equation} 
To begin with, using Proposition~\ref{regularityPt}, the functions $g_k$ belong to $C^3_b(\mathbb{R})$, and for any $i\in\{0,1,2\},$ for all $y\in\mathbb{R},$ $g_k^{(i)}(y)$ vanishes as $k$ goes to infinity.
Using again Proposition~\ref{regularityPt}, we see that  for all $i\in\{0,1,2,3\},$ $||g_k^{(i)}||_\infty$ is uniformly bounded in~$k$.
It follows that each sequence $(g_k^{(i)})_k$, $i\in\{0,1,2\},$ is uniformly equicontinuous and thus converges to zero uniformly on each compact interval. 

{We next show that this implies that also the sequences $ ( A^N g_k)_k $ and $ ( \bar A g_k )_k $ converge to zero uniformly on each compact interval. For $( \bar A g_k)_k , $ this is immediate, since $ \bar A $ is a local operator having continuous coefficients. For $ (A^N g_k )_k , $ it follows from the fact that $ A^N g_k ( x) \to 0 $ as $ k \to \infty $ for each fixed $x$ and the fact that by Lemma \ref{ANprime} given below, this sequence is uniformly (in $k$, for fixed $N$) equicontinuous on each compact.

We are now able to conclude. The sequence $(X^N_{t-s_k})_k$ is almost surely bounded by $\underset{0\leq r\leq t}{\sup}|X^N_r|$ which is finite almost surely by Lemma~\ref{controlmoments}.$(iii)$. Hence, almost surely as $ k \to \infty, $ $\bar Ag_k(X^N_{t-s_k})\to 0$ and $ A^Ng_k(X^N_{t-s_k})\to 0. $  

We now apply dominated convergence to prove \eqref{eq:exp}. Using that by Proposition \ref{derivegenerateur},
 for all $g\in C_b^3(\r)$ and $x\in\r$, 
$$|\bar{A}g(x)|\leq \constantetruc||g||_{2,\infty}(1+x^2), $$
% \mbox{ and } |A^Ng(x)|\leq \constantetruc||g||_{2,\infty}(1+x^2),$$
we can bound the expression in the first expectation by
$$C||g_k||_{2 , \infty }\left(1+(\underset{0\leq r\leq t}{\sup}|X^N_r|)^2\right)\leq 2C\left(\underset{0\leq r\leq t}{\sup}||\bar P_r g ||_{2,\infty}\right)\left(1+(\underset{0\leq r\leq t}{\sup}|X^N_r|)^2\right),$$
whose expectation is finite thanks to Lemma \ref{controlmoments}$(iii).$ The same arguments work for $A^N$. This implies that  \eqref{termenumero2} vanishes as $ k \to \infty,  $ and this concludes the proof. }
\end{proof}
We now prove the missing lemma 
\begin{lem}
\label{ANprime}
For all $g\in C_b^2(\r)$  and any $ M > 0, $ 
$$ \sup_{ x  \in [- M, M]} | \left(A^Ng\right)' ( x) | \leq\constantetruc_N {\| g\|_{2,\infty}}\left(1+M^2\right),$$
for some constant $\constantetruc_N>0$ that can depend on~$N,$ but not on $M.$ 
\end{lem}

\begin{proof}
We have
\begin{multline*}
\left(A^Ng\right)'(x)={-\alpha g'(x)-\alpha xg''(x)}+Nf'(x)\esp{g\left(x+\frac{U}{\sqrt{N}}\right)-g(x)}
\\+Nf(x)\esp{g'\left(x+\frac{U}{\sqrt{N}}\right)-g'(x)} .  
\end{multline*}
Since 
$$ \esp{\left| g\left(x+\frac{U}{\sqrt{N}}\right)-g(x)\right| } \le \frac{\| g'\|_\infty}{\sqrt{N}} \esp{ |U|} , \; \esp{|g'\left(x+\frac{U}{\sqrt{N}}\right)-g'(x)|} \le  \frac{\| g''\|_\infty}{\sqrt{N}} \esp{ |U|},$$ 
we obtain
$$\sup_{ x  \in [- M, M]} | \left(A^Ng\right)' ( x) | \leq |\alpha|\| g\|_{2,\infty}(1+M)+\sqrt N(|f'(x)|\vee |f(x)|)\| g\|_{2,\infty}\esp{|U|}.$$

Assumption \ref{hyp2} implies that $|f'(x)|\leq m_1 C\sqrt{1+x^2}$ for all $x.$ Together with the sub-quadraticity of $f,$ this concludes the proof.
\end{proof}

\subsection{Existence and uniqueness of the process $\left(X_t^N\right)_t$}\label{existenceXN}

\begin{prop}
\label{wellposed}
If Assumptions~\ref{hyp1f} and~\ref{hyp1control} hold, the equation~\eqref{XtNdefinition} admits a unique non-exploding strong solution.
\end{prop}

\begin{proof}
It is well known that if $f$ is bounded, there is a unique strong solution of~\eqref{XtNdefinition} (see Theorem~IV.9.1 of~\cite{ikeda}). In the general case we reason in a similar way as in the proof of Proposition~2 in~\cite{nicolaseva}. Consider the solution $(X^{N,K}_t)_{t\in\r_+}$ of the equation~\eqref{XtNdefinition} where $f$ is replaced by $f_K:x\in\r\mapsto f(x)\wedge\underset{|y|\leq K}{\sup}f(y)$ for some $K\in\n^*$. Introduce moreover the stopping time $$\tau_K^N=\inf\left\{t\geq 0~:~\left|X_t^{N,K}\right|\geq K\right\}.$$Since for all $t\in\left[0,\tau_K^N\wedge\tau_{K+1}^N\right]$, $X_t^{N,K}=X_t^{N,K+1}$, we know that $\tau_K^N(\omega)\leq\tau_{K+1}^N(\omega)$ for all $\omega$. Then we can define $\tau^N$ as the non-decreasing limit of $\tau_K^N$. With a classical reasoning relying on It\^o's formula and Gr\"onwall's lemma, we can prove that
\begin{equation}
\label{controlXNKt}
{\underset{0\leq s\leq t}{\sup}\esp{\left(X^{N,K}_{s\wedge\tau_K^N}\right)^2}}\leq C_t\left(1+x^2\right),
\end{equation}
where $C_t>0$ does not depend on $K$. As a consequence, we know that almost surely, $\tau^N = + \infty.$ So we can simply define $X_t^N$ as the limit of $X_t^{N,K}$, as $K$ goes to infinity. Now we show that $X^N$ satisfies equation~\eqref{XtNdefinition}. Consider some $\omega\in\Omega$ and $t>0$, and choose $K$ such that $\tau_K^N(\omega)>t$. Then we know that for all $s\in[0,t]$, $X_s^N(\omega)=X_s^{N,K}(\omega)$ and $f(X_{s-}^N(\omega))=f_K(X_{s-}^{N,K}(\omega)).$ Moreover, as $X^{N,K}(\omega)$ satisfies equation~\eqref{XtNdefinition} with $f$ replaced by $f_K$, we know that $X^N(\omega)$ verifies equation~\eqref{XtNdefinition} on $[0,t]$. This holds for all $t>0.$ As a consequence, we know that $X^N$ satisfies equation~\eqref{XtNdefinition}. This proves the existence of a strong solution. The uniqueness is a consequence of the uniqueness of strong solutions of~\eqref{XtNdefinition}, if we replace $f$ by $f_K$ in~\eqref{XtNdefinition}, and of the fact that any strong solution $(Y_t^N)_t$ equals necessarily $(X^{N,K}_t)_t$ on $[0,\tau^N_K]$.
\end{proof}

\subsection{Proof of Lemma~\ref{controlmoments}}
\begin{proof}

We begin with the proof of $(i)$. Let $\Phi(x)=x^2$ and $A^{N}$ be the extended generator of $(X^{N}_{t})_{t\geq 0}$. One can note that, applying Fatou's lemma to the inequality~\eqref{controlXNKt}, one obtains for all $t\geq 0,{\underset{0\leq s\leq t}{\sup}\esp{(X^N_s)^2}}$ is finite. As a consequence $\Phi\in\mathcal{D}'(A^{N})$ (in the sense of Definition~\ref{definitiongenerateur}). And, {recalling that $\mu$ is centered and that $\sigma^2:=\int_{\mathbb{R}} u^2d\mu(u)$, we have} for all $x\in\r,$
\begin{align*}
A^{N}\Phi(x)=&{-\alpha x\Phi'(x) + Nf(x)\int_{\mathbb{R}}\left[\Phi(x+\frac{u}{\sqrt{N}})-\Phi(x)\right]d\mu(u)}\\
=&{-2\alpha x^2 + Nf(x)\int_{\mathbb{R}}\left[2x\frac{u}{\sqrt{N}} + \frac{u^2}{N}\right]d\mu(u)}\\
=&-2\alpha \Phi(x)+\sigma^2f(x)\leq -2\alpha\Phi(x) + \sigma^2\left(L|x|+\sqrt{f(0)}\right)^2\\
\leq&(\sigma^2L^2-2\alpha)\Phi(x)+2\sigma^2L|x|\sqrt{f(0)}+\sigma^2f(0).
\end{align*}
Let $\eps>0$ be fixed, and $\eta_\eps=2\sigma^2L\sqrt{f(0)}/\eps$. Using that, for every $x\in\r, |x|\leq x^2/\eta_\eps+\eta_\eps,$ we have
\begin{equation}
\label{phicd}
A^{N}\Phi(x)\leq c_\eps\Phi(x) + d_\eps,
\end{equation}
with $c_\eps = \sigma^2L^2-2\alpha +\eps$ and $d_\eps = O(1/\eps).$ Let us assume that $c_\eps\neq 0$, possibly by reducing $\eps>0$. Considering $Y_t^{N}:=e^{-c_\eps t}\Phi(X^{N}_{t}),$ by It\^o's formula,
\begin{align*}
dY^{N}_t=&-c_\eps e^{-c_\eps t}\Phi(X^{N}_{t})dt + e^{-c_\eps t}d\Phi(X^{N}_{t})\\
=& -c_\eps e^{-c_\eps t}\Phi(X^{N}_{t})dt + e^{-c_\eps t}A^{N} \Phi(X^{N}_{t})dt + e^{-c_\eps t}dM_t,
\end{align*}
{where, denoting by $\tilde \pi_j(dt,dx,du) := \pi_j(dt,dx,du)-dtdxd\mu(u)$ the compensated measure of $\pi_j$ ($1\leq j\leq N$), $(M_t)_{t\geq 0}$ is the $\mathbb{P}_x-$local martingale defined as
$$M_t = \sum_{j=1}^N\int_{[0,t]\times\mathbb{R}_+\times\mathbb{R}}\left[\Phi(X^N_{s-}+\frac{u}{\sqrt{N}})-\Phi(x)\right]\uno{z\leq f(X^N_{s-})}d\tilde\pi_j(s,z,u).$$

One can note that, since $\underset{0\leq s\leq t}{\sup}\esp{(X^N_s)^2}$ is finite for any $t\geq 0,$ $(M_t)_{t\geq 0}$ is a locally square integrable $\mathbb{P}_x-$local martingale, and as a consequence, it is a $\mathbb{P}_x-$martingale.}

Using~\eqref{phicd}, we obtain
$$dY^{N}_t\leq d_\eps e^{-c_\eps t} dt +e^{-c_\eps t}dM_t,$$
implying
$$\espcc{x}{Y^{N}_t}\leq \espcc{x}{Y^{N}_0}+\frac{d_\eps}{c_\eps} e\ll(^{-c_\eps t}+1\rr).$$
One deduces
\begin{equation}
\label{mNtinftaukbound}
\espcc{x}{\left(X^{N}_{t}\right)^2}\leq x^2e^{(\sigma^2L^2-2\alpha+\eps)t}+\frac{\constantetruc}{\eps}\ll(e^{(\sigma^2L^2-2\alpha+\eps)t}+1\rr),
\end{equation}
for some constant $\constantetruc>0$ independent of $t,\eps,N.$

\noindent
The proof of $(ii)$ is analogous and therefore omitted.

\noindent
Now we prove $(iii).$
From
$$
X_t^N= X_0^N-\alpha\int_0^t X_s^Nds+\frac{1}{\sqrt{N}}\sum_{j=1}^N\int_{]0,t]\times\r_+\times\r}u\uno{z\leq f(X^N_{s-})}d\pi_j(s,z,u),
$$
we deduce 
\begin{multline}\label{bdg}
\left(\underset{0\leq s\leq t}{\sup}\left|X_t^N\right|\right)^2\leq 3\left(X_0^N\right)^2+3\alpha^2t\int_0^t(X_s^N)^2ds \\
+3\sum_{j=1}^N\left(\underset{0\leq s\leq t}{\sup}\left|\int_{]0,s]\times\r_+\times\r}u\uno{z\leq f(X^N_{r-})}d\pi_j(r,z,u)\right|\right)^2.
\end{multline}
Applying Burkholder-Davis-Gundy inequality to the last term above in~\eqref{bdg}, we can bound its expectation by
\begin{align}
3N\esp{\displaystyle\int_{]0,t]\times\r_+\times\r}u^2\uno{z\leq f(X^N_{s-})}d\pi_j(s,z,u)}&\leq 3N\sigma^2\int_0^t\esp{f(X^N_{s-})}ds\nonumber\\
&\leq  3N\sigma^2\constantetruc\int_0^t\left(1+\esp{(X^N_s)^2}\right)ds.\label{bdgterm}
\end{align}

Now, bounding {the expectation of}~\eqref{bdg} by~\eqref{bdgterm}, and using point $(i)$ of the lemma we conclude the proof of $(iii)$.

{The assertion  $(iv)$ can be proved in classical way, applying It\^o's formula and Gr\"onwall's lemma. Let us explain how to prove this property for the process $X^N.$ The proof for $\bar X$ is similar.
According to It\^o's formula, for every $t\geq 0,$
\begin{multline*}
(X^N_t)^4
= (X^N_0)^4-4\alpha\int_0^t (X^N_s)^4ds \\
+ \sum_{j=1}^N\int_{[0,t]\times\mathbb{R}_+\times\mathbb{R}}\left[(X^N_{s-}+\frac{u}{\sqrt{N}})^4-(X^N_{s-})^4\right]\uno{z\leq f(X^N_{s-})}d\pi_j(s,z,u)\\
\leq (X^N_0)^4+\sum_{j=1}^N\int_{[0,t]\times\mathbb{R}_+\times\mathbb{R}}\left[(X^N_{s-}+\frac{u}{\sqrt{N}})^4-(X^N_{s-})^4\right]\uno{z\leq f(X^N_{s-})}d\pi_j(s,z,u).
\end{multline*}

Let us recall that $u$ is centered and has a finite fourth moment, and that $f$ is subquadratic. Introducing the stopping times $\tau^N_K:=\inf\{t>0:|X^N_t|>K\}$ for $K>0,$ and $u^N_K(t) := \esp{(X^N_{t\wedge\tau^N_K})^4},$ it follows from the above that for all $t\geq 0,$
$$u^N_K(t) \leq C + Ct + C\int_0^t u^N_K(s)ds,$$
where $C$ is a constant independent of $t,N$ and $K$. This implies that for all $t\geq 0,$
$$\underset{N\in\mathbb{N}^*}{\sup}~\underset{K>0}{\sup}~\underset{0\leq s\leq t}{\sup} u^N_K(s)<\infty.$$

Consequently, the stopping times $\tau^N_K$ tend to infinity as $K$ goes to infinity, and Fatou's lemma allows to conclude.}

{{We finally prove~$(v)$. Indeed, by  It\^o's isometry and Jensen's inequality, for all $0\leq s\leq t \le T,$ using the sub-quadraticity of $f$ and $(i),$
\begin{align*}
\espccc{x}{N}{(X^N_t-X^N_s)^2} =& \mathbb{E}_x \left[ \left(-\alpha\int_s^t X^N_rdr + \frac{1}{\sqrt{N}}\sum_{j=1}^N\int_{]s,t]\times\mathbb{R}_+\times\mathbb{R}}u\uno{z\leq f(X^N_{r-})}d\pi_j(r,z,u)\right)^2\right] \\
\leq& 2\alpha^2(t-s)\int_s^t\espccc{x}{N}{(X^N_r)^2}dr + 2\sigma^2\int_s^t\espccc{x}{N}{f(X^N_r)}dr\\
\leq& 2\alpha^2 C_{t}(1+x^2)(t-s)^2 + 2\sigma^2C_{t}(1+x^2)(t-s)\\
\leq & C_T(t-s) (1 + x^2 ).
\end{align*}
This proves that $X^N$ satisfies hypothesis~$(v)$. A similar computation holds true for $\bar X.$}}
\end{proof}
\subsection{Proof of Proposition~\ref{regularityPt}}
\begin{proof}

To begin with, we use Theorem~1.4.1 of~\cite{kunita2} to prove that the flow associated to the SDE~\eqref{sde} admits a modification which is $C^3$ with respect to the initial condition $x$ (see also Theorem~4.6.5 of~\cite{kunita}). Indeed the local characteristics of the flow are given by $$b(x,t)=-\alpha x~\textrm{ and }~a(x,y,t)=\sigma^2\sqrt{f(x)f(y)},$$ and, under Assumptions~\ref{hyp1f} and~\ref{hyp2}, they satisfy the conditions of Theorem~1.4.1 of~\cite{kunita2}:
\begin{itemize}
\item $\exists \constantetruc,\forall x,y,t,|b(x,t)|\leq \constantetruc(1+|x|)$ and $|a(x,y,t)|\leq \constantetruc(1+|x|)(1+|y|)$.
\item $\exists \constantetruc,\forall x,y,t,|b(x,t)-b(y,t)|\leq \constantetruc|x-y|$ and $|a(x,x,t)+a(y,y,t)-2a(x,y,t)|\leq \constantetruc|x-y|^2$.
\item $\forall 1\leq k\leq 4,1\leq l\leq 4-k,\frac{\partial^k}{\partial x^k}b(x,t)$ and $\frac{\partial^{k+l}}{\partial x^k\partial y^l}a(x,y,t)$ are bounded.
\end{itemize}

In the following, we consider the process $(\bar{X}_t^{(x)})_t, $ solution of the SDE~\eqref{sde} and satisfying $\bar{X}_0^{(x)}=x$. Then we can consider a modification of the flow~$\bar{X}^{(x)}_t$ which is $C^3$ with the respect to the initial condition $x=\bar{X}_0^{(x)}$. It is then sufficient to control the moment of the derivatives of $\bar{X}_t^{(x)}$ with respect to~$x$, since with those controls we will have
\begin{align}
\bar{P}_tg(x)=&\esp{g\left(\bar{X}_t^{(x)}\right)}, \; \; 
\left(\bar{P}_tg\right)'(x)=\esp{\deriv{1}{t}g'\left(\bar{X}_t^{(x)}\right)},\nonumber\\
\left(\bar{P}_tg\right)''(x)=&\esp{\deriv{2}{t}g'\left(\bar{X}_t^{(x)}\right)+\left(\deriv{1}{t}\right)^2g''\left(\bar{X}_t^{(x)}\right)},\nonumber\\
\left(\bar{P}_tg\right)'''(x)=&\esp{\deriv{3}{t}g'\left(\bar{X}_t^{(x)}\right)+3\deriv{2}{t}\cdot\deriv{1}{t}g''\left(\bar{X}_t^{(x)}\right)+\left(\deriv{1}{t}\right)^3g'''\left(\bar{X}_t^{(x)}\right)}\label{Ptgd3}.
\end{align}
We start with the representation
$$\bar{X}_t^{(x)}=xe^{-\alpha t}+\sigma\int_0^te^{-\alpha (t-s)}\sqrt{f\left(\bar{X}_s^{(x)}\right)}dB_s.$$
This implies
\begin{equation}
\label{derive1}
\deriv{1}{t}=e^{-\alpha t}+\sigma\displaystyle\int_0^te^{-\alpha (t-s)}\deriv{1}{s}\derivf{1}\left(\bar{X}_s^{(x)}\right)dB_s.
\end{equation}
Writing $U_t=e^{\alpha t}\deriv{1}{t}$ and 
\begin{equation}\label{eq:mt}
 M_t = \int_0^t \sigma \derivf{1}\left(\bar{X}_s^{(x)}\right)dB_s , 
\end{equation}
we obtain $U_t=1+\int_0^t U_s dM_s,$
whence
\begin{equation}\label{Utcont}
U_t = \exp \left( M_t - \frac12 <M>_t\right). 
\end{equation}
Notice that this implies $ U_t > 0 $ almost surely, whence $ \deriv{1}{t} > 0 $ almost surely. Hence
$$ U_t^p = e^{ p M_t - \frac{p}{2} <M>_t} = \exp \left( p M_t - \frac12 p^2 <M>_t \right) e^{ \frac12 p (p-1) <M>_t} = {\mathcal E}(M)_te^{ \frac12 p (p-1) <M>_t} .$$
Since $ \derivf{1} $ is bounded, $M_t $ is a martingale, thus ${\mathcal E}(M)$ is an exponential martingale with expectation $1,$ implying that 
\begin{equation}\label{eq:utp}
 \E U_t^p \le e^{ \frac12 p (p-1) \sigma^2 m^2_1 t },
\end{equation}
where $m_1$ is the bound of $(\sqrt{f})'$ introduced in Assumption~\ref{hyp2}.
In particular we have
\begin{equation}
\label{deriv123}
\esp{\left(\deriv{1}{t}\right)^2}\leq e^{(\sigma^2m_1^2-2\alpha)t}\textrm{ and }\esp{\left|\deriv{1}{t}\right|^3}\leq e^{(3\sigma^2m_1^2-3\alpha)t}.
\end{equation}
%\begin{equation}
%\label{deriv12}
%\esp{\left(\deriv{1}{t}\right)^2}\leq e^{(\sigma^2m_1^2-2\alpha)t},
%\end{equation}
%and
%\begin{equation}
%\label{deriv13}
%\esp{\left|\deriv{1}{t}\right|^3}\leq e^{(3\sigma^2m_1^2-3\alpha)t}.
%\end{equation}
Differentiating~\eqref{derive1} with respect to $x$, we obtain
\begin{equation}
\label{derive2}
\deriv{2}{t}=\sigma\int_0^te^{-\alpha(t-s)}\left[\deriv{2}{s}\derivf{1}\left(\bar{X}_s^{(x)}\right)+\left(\deriv{1}{s}\right)^2\derivf{2}\left(\bar{X}_s^{(x)}\right)\right]dB_s.
\end{equation}
We introduce $V_t=\deriv{2}{t}e^{\alpha t}$ and deduce from this that
\begin{align*}
V_t=&\sigma\int_0^t\left[V_s\derivf{1}\left(\bar{X}_s^{(x)}\right)+e^{-\alpha s}U_s^2\derivf{2}\left(\bar{X}_s^{(x)}\right)\right]dB_s,
\end{align*}
which can be rewritten as 
$$ d V_t = V_t d M_t + Y_t d B_t, V_0 = 0 , Y_t = \sigma e^{-\alpha t}U_t^2\derivf{2}\left(\bar{X}_t^{(x)}\right),$$
with $M_t$ as in \eqref{eq:mt}.
Applying It\^o's formula to $ Z_t:= V_t/U_t $ (recall that $ U_t > 0 $), we obtain
$$ d Z_t = \frac{Y_t}{U_t} d B_t - \frac{Y_t}{U_t} d < M , B>_t , $$ 
such that, by the precise form of $ Y_t$ and since $ Z_0 = 0, $  
$$ Z_t = \sigma\int_0^t e^{- \alpha s } U_s\derivf{2}\left(\bar{X}_s^{(x)}\right)d B_s - \sigma^2 \int_0^t e^{- \alpha s } U_s\derivf{2}\left(\bar{X}_s^{(x)} \right) \derivf{1}\left(\bar{X}_s^{(x)} \right)ds .$$ 
Using Jensen's inequality, \eqref{eq:utp} and Burkholder-Davis-Gundy inequality, for all $t \geq 0, $ 
\begin{multline}\label{zt4}
\esp{Z_t^4}\leq \constantetruc \left(\esp{\left(\displaystyle\int_0^te^{-\alpha s}U_s\derivf{2}\left(\bar{X}_s^{(x)}\right)dB_s\right)^4}\right. \\
\left.+\esp{\left(\int_0^te^{-\alpha s}U_s\derivf{1}\left(\bar{X}_s^{(x)} \right)\derivf{2}\left(\bar{X}_s^{(x)}\right)ds\right)^4}\right) \\
\leq \constantetruc \left(\esp{\left(\displaystyle\int_0^te^{-2\alpha s}U_s^2\derivf{2}\left(\bar{X}_s^{(x)}\right)^2ds\right)^2}\right. \\
\left.+\esp{\left(\int_0^te^{-\alpha s}U_s\derivf{1}\left(\bar{X}_s^{(x)} \right)\derivf{2}\left(\bar{X}_s^{(x)}\right)ds\right)^4}\right) \\
\leq \constantetruc\left(t+t^3\right)\int_0^t e^{-4\alpha s}\esp{U_s^4}ds\leq \constantetruc\left(t+t^3\right)\int_0^te^{(6\sigma^2m_1^2-4\alpha)s}ds\\
\leq  \constantetruc\left(t+t^3\right)\left(1+t+e^{(6\sigma^2m_1^2-4\alpha)t}\right)\leq C(t+t^4)e^{(6\sigma^2m_1^2-4\alpha)t}.
\end{multline}
We deduce that
\begin{multline*}
\esp{V_t^2}\leq \esp{Z_t^4}^{1/2}\esp{U_t^4}^{1/2}\leq C(t^{1/2}+t^2)e^{3\sigma^2m_1^2-2\alpha t}e^{3\sigma^2m_1^2t}\leq C(t^{1/2}+t^2)e^{6\sigma^2m_1^2-2\alpha t},
\end{multline*}
whence
\begin{equation}
\label{deriv22}
\esp{\left(\deriv{2}{t}\right)^2}\leq \constantetruc(t^{1/2}+t^2)e^{(6\sigma^2m_1^2-4\alpha)t}.
\end{equation}
Finally, differentiating~\eqref{derive2}, we get
\begin{align}
\deriv{3}{t}=\sigma\int_0^te^{-\alpha(t-s)}&\left[\deriv{3}{s}\derivf{1}\left(\bar{X}_s^{(x)}\right)+3\deriv{2}{s}\deriv{1}{s}\derivf{2}\left(\bar{X}_s^{(x)}\right)\right.\label{x3cont}\\
&~~\left.+\left(\deriv{1}{s}\right)^3\derivf{3}\left(\bar{X}_s^{(x)}\right)\right]dB_s.\nonumber
\end{align}
Introducing $W_t=e^{\alpha t}\deriv{3}{t}$, we obtain
$$W_t=\sigma\int_0^t\left[W_s\derivf{1}\left(\bar{X}_s^{(x)}\right)+3e^{-\alpha s}U_sV_s\derivf{2}\left(\bar{X}_s^{(x)}\right)+e^{-2\alpha s}U_s^3\derivf{3}\left(\bar{X}_s^{(x)}\right)\right]dB_s.$$
Once again we can rewrite this as 
$$ d W_t = W_t d M_t + Y_t' d B_t, W_0 = 0 ,$$
where 
$$  Y_t' = \sigma \left( 3e^{-\alpha t}U_tV_t\derivf{2}\left(\bar{X}_t^{(x)}\right) + e^{-2\alpha t}U_t^3\derivf{3}\left(\bar{X}_t^{(x)}\right) \right) ,  $$
whence, introducing $Z'_t=\frac{W_t}{U_t},$
$$    Z'_t = \int_0^t \frac{Y'_s}{U_s} d B_s - \int_0^t \frac{Y'_s}{U_s} d < M , B>_s.$$
As previously, we obtain,
\begin{align}
\esp{\left(Z'_t\right)^2}\leq& \constantetruc(1+t)\int_0^t\esp{\left(\frac{Y'_s}{U_s}\right)^2}ds\nonumber\\
\leq & \constantetruc(1+t)\int_0^t\left(e^{-2\alpha s}\esp{V_s^2}+e^{-4\alpha s}\esp{U_s^4}\right)ds\nonumber\\
\leq& \constantetruc(1+t)\int_0^t\left((s^{1/2}+s^2)e^{(6\sigma^2m_1^2-4\alpha)s}+e^{(6\sigma^2m_1^2-4\alpha)s}\right)ds\nonumber\\
\leq&\constantetruc(1+t^3)\int_0^t e^{(6\sigma^2m_1^2-4\alpha)s}dss\nonumber\\
\leq& \constantetruc(1+t^3)(1+t+e^{(6\sigma^2m_1^2-4\alpha)t}) \; \; 
\leq \constantetruc(1+t^4)\left(1+e^{(6\sigma^2m_1^2-4\alpha)t}\right).
\end{align}

As a consequence,
\begin{align*}
\esp{|W_t|}\leq& \esp{(Z'_t)^2}^{1/2}\esp{U_t^2}^{1/2}\leq \constantetruc(1+t^{2})\left(1+e^{(3\sigma^2m_1^2-2\alpha) t}\right)e^{\frac12\sigma^2m_1^2t}\\
\leq& \constantetruc(1+t^{2})\left(e^{\frac12\sigma^2m_1^2t}+e^{(\frac72\sigma^2m_1^2-2\alpha) t}\right),
\end{align*}

implying
\begin{equation}
\label{deriv31}
\esp{\left|\frac{\partial^3\bar{X}_t^{(x)}}{\partial^3x}\right|}\leq \constantetruc(1+t^{2})\left(e^{(\frac12\sigma^2m_1^2-\alpha)t}+e^{(\frac72\sigma^2m_1^2-3\alpha) t}\right).
\end{equation}

Finally, using Cauchy-Schwarz inequality, and inserting~\eqref{deriv123},~\eqref{deriv22} and~\eqref{deriv31} in~\eqref{Ptgd3},
$$\left|\left|\left(\bar{P}_tg\right)'''\right|\right|_\infty\leq \constantetruc||g||_{3,\infty}(1+t^2)\left(e^{(\frac12\sigma^2m_1^2-\alpha)t}+e^{2(\sigma^2m_1^2-\alpha)t}+e^{(\frac72\sigma^2m_1^2-3\alpha)t}\right),$$
which proves
{the first assertion of the  proposition. The proof of the second assertion, equation \eqref{eq:qt2}, follows similarly. Finally to prove the third assertion,
we first study the regularity of the first derivative. Notice that $t\mapsto \frac{\partial \bar X_t^{(x)}}{\partial x}$ is almost surely continuous by equation~\eqref{derive1}. Now take any sequence $ t_n \to t .$ By \eqref{deriv123}, the family of random variables $ \left\{  \frac{\partial \bar X_{t_n}^{(x)}}{\partial x} g' ( \bar X_{t_n}^{(x)}), n \geq 1  \right\}$ is uniformly integrable. As a consequence, the second formula in \eqref{Ptgd3} implies that $ (\bar P_{t_n}g)'(x) \to (\bar P_{t}g)'(x) $ as $n \to \infty ,$ whence the desired continuity. The argument is similar for the second derivative, using~\eqref{derive2} and~\eqref{deriv22}.}
That concludes the proof.
\end{proof}

\nocite{clinet_statistical_2017}
\bibliography{convergenceZN_barZ}

\begin{thebibliography}{39}
% BibTex style file: imsart-nameyear.bst, 2017-11-03
% Default style options (sort=1,type=nameyear).
% Used options (sort=1,type=nameyear).

\bibitem[\protect\citeauthoryear{A\"it-Sahalia, Cacho-Diaz and
  Laeven}{2015}]{aitsahalia}
\begin{barticle}[author]
\bauthor{\bsnm{A\"it-Sahalia},~\bfnm{Yacine}\binits{Y.}},
  \bauthor{\bsnm{Cacho-Diaz},~\bfnm{Julio}\binits{J.}} \AND
  \bauthor{\bsnm{Laeven},~\bfnm{Roger J.~A.}\binits{R.~J.~A.}}
(\byear{2015}).
\btitle{{Modeling financial contagion using mutually exciting jump processes}}.
\bjournal{{Journal of Financial Economics}}
\bvolume{117}
\bpages{585-606}.
\end{barticle}
\endbibitem

\bibitem[\protect\citeauthoryear{Bacry and Muzy}{2016}]{bacrymuzy}
\begin{barticle}[author]
\bauthor{\bsnm{Bacry},~\bfnm{E.}\binits{E.}} \AND
  \bauthor{\bsnm{Muzy},~\bfnm{J.~F.}\binits{J.~F.}}
(\byear{2016}).
\btitle{{Second order statistics characterization of {H}awkes processes and
  non-parametric estimation}}.
\bjournal{{Trans. in Inf. Theory}}
\bvolume{2}.
\end{barticle}
\endbibitem

\bibitem[\protect\citeauthoryear{Bauwens and Hautsch}{2009}]{bauwens}
\begin{bbook}[author]
\bauthor{\bsnm{Bauwens},~\bfnm{Luc}\binits{L.}} \AND
  \bauthor{\bsnm{Hautsch},~\bfnm{Nikolaus}\binits{N.}}
(\byear{2009}).
\btitle{{Modelling financial high frequency data using point processes}}.
\bpublisher{{Springer Berlin Heidelberg}}.
\end{bbook}
\endbibitem

\bibitem[\protect\citeauthoryear{Billingsley}{1999}]{billingsley}
\begin{bbook}[author]
\bauthor{\bsnm{Billingsley},~\bfnm{Patrick}\binits{P.}}
(\byear{1999}).
\btitle{{Convergence of Probability Measures}},
\bedition{Second} ed.
\bpublisher{{Wiley Series In Probability And Statistics}}.
\end{bbook}
\endbibitem

\bibitem[\protect\citeauthoryear{Br\'emaud and Massouli\'e}{1996}]{bremaud}
\begin{barticle}[author]
\bauthor{\bsnm{Br\'emaud},~\bfnm{Pierre}\binits{P.}} \AND
  \bauthor{\bsnm{Massouli\'e},~\bfnm{Laurent}\binits{L.}}
(\byear{1996}).
\btitle{{Stability of Nonlinear {H}awkes Processes}}.
\bjournal{{The Annals of Probability}}
\bvolume{24}
\bpages{1563-1588}.
\end{barticle}
\endbibitem

\bibitem[\protect\citeauthoryear{Carmona, Delarue and Lacker}{2016}]{delarue}
\begin{barticle}[author]
\bauthor{\bsnm{Carmona},~\bfnm{Ren\'e}\binits{R.}},
  \bauthor{\bsnm{Delarue},~\bfnm{François}\binits{F.}} \AND
  \bauthor{\bsnm{Lacker},~\bfnm{Daniel}\binits{D.}}
(\byear{2016}).
\btitle{{Mean field games with common noise}}.
\bjournal{{Ann. Probab.}}
\bvolume{44}
\bpages{3740--3803}.
\bdoi{10.1214/15-AOP1060}
\end{barticle}
\endbibitem

\bibitem[\protect\citeauthoryear{Clinet and
  Yoshida}{2017}]{clinet_statistical_2017}
\begin{barticle}[author]
\bauthor{\bsnm{Clinet},~\bfnm{Simon}\binits{S.}} \AND
  \bauthor{\bsnm{Yoshida},~\bfnm{Nakahiro}\binits{N.}}
(\byear{2017}).
\btitle{Statistical inference for ergodic point processes and application to
  {Limit} {Order} {Book}}.
\bjournal{Stochastic Processes and their Applications}
\bvolume{127}
\bpages{1800--1839}.
\bdoi{10.1016/j.spa.2016.09.014}
\end{barticle}
\endbibitem

\bibitem[\protect\citeauthoryear{{Costa} et~al.}{2018}]{costa}
\begin{barticle}[author]
\bauthor{\bsnm{{Costa}},~\bfnm{Manon}\binits{M.}},
  \bauthor{\bsnm{{Graham}},~\bfnm{Carl}\binits{C.}},
  \bauthor{\bsnm{{Marsalle}},~\bfnm{Laurence}\binits{L.}} \AND
  \bauthor{\bsnm{{Tran}},~\bfnm{Viet~Chi}\binits{V.~C.}}
(\byear{2018}).
\btitle{{Renewal in Hawkes processes with self-excitation and inhibition}}.
\bjournal{arXiv e-prints}
\bpages{arXiv:1801.04645}.
\end{barticle}
\endbibitem

\bibitem[\protect\citeauthoryear{Daley and Vere-Jones}{2003}]{daley}
\begin{bbook}[author]
\bauthor{\bsnm{Daley},~\bfnm{D.~J.}\binits{D.~J.}} \AND
  \bauthor{\bsnm{Vere-Jones},~\bfnm{D.}\binits{D.}}
(\byear{2003}).
\btitle{{An Introduction to the Theory of Point Processes: Volume I: Elementary
  Theory and Methods}},
\bedition{Second} ed.
\bpublisher{{Springer}}.
\end{bbook}
\endbibitem

\bibitem[\protect\citeauthoryear{Davis}{1993}]{davis}
\begin{bbook}[author]
\bauthor{\bsnm{Davis},~\bfnm{M.~H.~A.}\binits{M.~H.~A.}}
(\byear{1993}).
\btitle{{{M}arkov Models and Optimization}},
\bedition{First} ed.
\bpublisher{{Springer Science+Business Media Dordrecht}}.
\end{bbook}
\endbibitem

\bibitem[\protect\citeauthoryear{Dawson and
  Vaillancourt}{1995}]{dawson_stochastic_1995}
\begin{barticle}[author]
\bauthor{\bsnm{Dawson},~\bfnm{Donald}\binits{D.}} \AND
  \bauthor{\bsnm{Vaillancourt},~\bfnm{Jean}\binits{J.}}
(\byear{1995}).
\btitle{Stochastic {McKean}-{Vlasov} equations}.
\bjournal{Nonlinear Differential Equations and Applications NoDEA}
\bvolume{2}
\bpages{199--229}.
\bdoi{10.1007/BF01295311}
\end{barticle}
\endbibitem

\bibitem[\protect\citeauthoryear{Delattre, Fournier and
  Hoffmann}{2016}]{highdimensional}
\begin{barticle}[author]
\bauthor{\bsnm{Delattre},~\bfnm{Sylvain}\binits{S.}},
  \bauthor{\bsnm{Fournier},~\bfnm{Nicolas}\binits{N.}} \AND
  \bauthor{\bsnm{Hoffmann},~\bfnm{Marc}\binits{M.}}
(\byear{2016}).
\btitle{{Hawkes processes on large networks}}.
\bjournal{{Ann. Appl. Probab.}}
\bvolume{26}
\bpages{216--261}.
\bdoi{10.1214/14-AAP1089}
\end{barticle}
\endbibitem

\bibitem[\protect\citeauthoryear{Ditlevsen and L\"ocherbach}{2017}]{evasusanne}
\begin{barticle}[author]
\bauthor{\bsnm{Ditlevsen},~\bfnm{Susanne}\binits{S.}} \AND
  \bauthor{\bsnm{L\"ocherbach},~\bfnm{Eva}\binits{E.}}
(\byear{2017}).
\btitle{{Multi-class Oscillating Systems of Interacting Neurons}}.
\bjournal{{Stochastic Processes and their Applications}}
\bvolume{127}
\bpages{1840-1869}.
\end{barticle}
\endbibitem

\bibitem[\protect\citeauthoryear{Ethier and Kurtz}{2005}]{ethier}
\begin{bbook}[author]
\bauthor{\bsnm{Ethier},~\bfnm{Stewart}\binits{S.}} \AND
  \bauthor{\bsnm{Kurtz},~\bfnm{Thomas}\binits{T.}}
(\byear{2005}).
\btitle{{{M}arkov Processes. Characterization and Convergence}}.
\bpublisher{{Wiley Series In Probability And Statistics}}.
\end{bbook}
\endbibitem

\bibitem[\protect\citeauthoryear{Fournier and L\"ocherbach}{2016}]{nicolaseva}
\begin{barticle}[author]
\bauthor{\bsnm{Fournier},~\bfnm{Nicolas}\binits{N.}} \AND
  \bauthor{\bsnm{L\"ocherbach},~\bfnm{Eva}\binits{E.}}
(\byear{2016}).
\btitle{{On a toy model of interacting neurons}}.
\bjournal{{Annales de l'Institut Henri Poincar\'e - Probabilit\'es et
  Statistiques}}
\bvolume{52}
\bpages{1844-1876}.
\end{barticle}
\endbibitem

\bibitem[\protect\citeauthoryear{Graham}{2019}]{graham_regenerative_2019}
\begin{barticle}[author]
\bauthor{\bsnm{Graham},~\bfnm{Carl}\binits{C.}}
(\byear{2019}).
\btitle{Regenerative properties of the linear {H}awkes process with unbounded
  memory}.
\bjournal{arXiv:1905.11053 [math, stat]}.
\end{barticle}
\endbibitem

\bibitem[\protect\citeauthoryear{Gr\"un, Diedsmann and Aertsen}{2010}]{grun}
\begin{bbook}[author]
\bauthor{\bsnm{Gr\"un},~\bfnm{S.}\binits{S.}},
  \bauthor{\bsnm{Diedsmann},~\bfnm{M.}\binits{M.}} \AND
  \bauthor{\bsnm{Aertsen},~\bfnm{A.~M.}\binits{A.~M.}}
(\byear{2010}).
\btitle{{Analysis of Parallel Spike Trains}}.
\bpublisher{Rotter, Springer series in computational neurosciences}.
\end{bbook}
\endbibitem

\bibitem[\protect\citeauthoryear{Hawkes}{1971}]{hawkes}
\begin{barticle}[author]
\bauthor{\bsnm{Hawkes},~\bfnm{A.~G.}\binits{A.~G.}}
(\byear{1971}).
\btitle{{Spectra of some self-exciting and mutually exciting point processes}}.
\bjournal{Biometrika}
\bvolume{58}
\bpages{83-90}.
\end{barticle}
\endbibitem

\bibitem[\protect\citeauthoryear{Hawkes and Oakes}{1974}]{hawkesoakes}
\begin{barticle}[author]
\bauthor{\bsnm{Hawkes},~\bfnm{Alan~G.}\binits{A.~G.}} \AND
  \bauthor{\bsnm{Oakes},~\bfnm{David}\binits{D.}}
(\byear{1974}).
\btitle{{A Cluster Process Representation of a Self-Exciting Process}}.
\bjournal{{Journal of Applied Probability}}
\bvolume{11}
\bpages{493--503}.
\end{barticle}
\endbibitem

\bibitem[\protect\citeauthoryear{Helmstetter and Sornette}{2002}]{helmstetter}
\begin{barticle}[author]
\bauthor{\bsnm{Helmstetter},~\bfnm{A}\binits{A.}} \AND
  \bauthor{\bsnm{Sornette},~\bfnm{Didier}\binits{D.}}
(\byear{2002}).
\btitle{{Subcritical and supercritical regimes in epidemic models of earthquake
  aftershocks}}.
\bjournal{{Journal of Geophysical Research}}
\bvolume{107}.
\bdoi{10.1029/2001JB001580}
\end{barticle}
\endbibitem

\bibitem[\protect\citeauthoryear{Hewlett}{2006}]{hewlett}
\begin{bmisc}[author]
\bauthor{\bsnm{Hewlett},~\bfnm{P.}\binits{P.}}
(\byear{2006}).
\btitle{{Clustering of order arrivals, price impact and trade path
  optimisation}}.
\bnote{{In Workshop on Financial Modeling with Jump processes. Ecole
  Polytechnique}}.
\end{bmisc}
\endbibitem

\bibitem[\protect\citeauthoryear{Ikeda and Watanabe}{1989}]{ikeda}
\begin{bbook}[author]
\bauthor{\bsnm{Ikeda},~\bfnm{Nobuyuki}\binits{N.}} \AND
  \bauthor{\bsnm{Watanabe},~\bfnm{Shinzo}\binits{S.}}
(\byear{1989}).
\btitle{{Stochastic Differential Equations and Diffusion Processes}},
\bedition{Second} ed.
\bpublisher{{North-Holland Publishing Company}}.
\end{bbook}
\endbibitem

\bibitem[\protect\citeauthoryear{Jacod and Shiryaev}{2003}]{jacod}
\begin{bbook}[author]
\bauthor{\bsnm{Jacod},~\bfnm{Jean}\binits{J.}} \AND
  \bauthor{\bsnm{Shiryaev},~\bfnm{Albert~N}\binits{A.~N.}}
(\byear{2003}).
\btitle{{Limit Theorems for Stochastic Processes}},
\bedition{Second} ed.
\bpublisher{{Springer-Verlag BerlinHeidelberg NewYork}}.
\end{bbook}
\endbibitem

\bibitem[\protect\citeauthoryear{Kallenberg}{1997}]{kallenberg_foundations_1997}
\begin{bbook}[author]
\bauthor{\bsnm{Kallenberg},~\bfnm{Olav}\binits{O.}}
(\byear{1997}).
\btitle{Foundations of {Modern} {Probability}}.
\bseries{Probability and {Its} {Applications}}.
\bpublisher{Springer-Verlag}, \baddress{New York}.
\bdoi{10.1007/b98838}
\end{bbook}
\endbibitem

\bibitem[\protect\citeauthoryear{Khasminskii}{2012}]{khasminskii_stochastic_2012}
\begin{bbook}[author]
\bauthor{\bsnm{Khasminskii},~\bfnm{Rafail}\binits{R.}}
(\byear{2012}).
\btitle{Stochastic stability of differential equations},
\bedition{Second} ed.
\bpublisher{Springer}.
\end{bbook}
\endbibitem

\bibitem[\protect\citeauthoryear{Kunita}{1986}]{kunita2}
\begin{bmisc}[author]
\bauthor{\bsnm{Kunita},~\bfnm{Hiroshi}\binits{H.}}
(\byear{1986}).
\btitle{{Lectures on Stochastic Flows and Applications for the {I}ndian
  {I}nstitute Of {S}cience {B}angalore}}.
\end{bmisc}
\endbibitem

\bibitem[\protect\citeauthoryear{Kunita}{1990}]{kunita}
\begin{bbook}[author]
\bauthor{\bsnm{Kunita},~\bfnm{Hiroshi}\binits{H.}}
(\byear{1990}).
\btitle{{Stochastic flows and stochastic differential equations}}.
\bpublisher{{Cambridge University Press}}.
\end{bbook}
\endbibitem

\bibitem[\protect\citeauthoryear{Kurtz and Xiong}{1999}]{kurtz_particle_1999}
\begin{barticle}[author]
\bauthor{\bsnm{Kurtz},~\bfnm{Thomas~G.}\binits{T.~G.}} \AND
  \bauthor{\bsnm{Xiong},~\bfnm{Jie}\binits{J.}}
(\byear{1999}).
\btitle{Particle representations for a class of nonlinear {SPDEs}}.
\bjournal{Stochastic Processes and their Applications}
\bvolume{83}
\bpages{103--126}.
\bdoi{10.1016/S0304-4149(99)00024-1}
\end{barticle}
\endbibitem

\bibitem[\protect\citeauthoryear{Lu and Abergel}{2018}]{abergel}
\begin{barticle}[author]
\bauthor{\bsnm{Lu},~\bfnm{Xiaofei}\binits{X.}} \AND
  \bauthor{\bsnm{Abergel},~\bfnm{Frédéric}\binits{F.}}
(\byear{2018}).
\btitle{High dimensional {H}awkes processes for limit order books Modelling,
  empirical analysis and numerical calibration}.
\bjournal{Quantitative Finance}
\bpages{1-16}.
\bdoi{10.1080/14697688.2017.1403142}
\end{barticle}
\endbibitem

\bibitem[\protect\citeauthoryear{Meyn and Tweedie}{1993}]{meyn}
\begin{barticle}[author]
\bauthor{\bsnm{Meyn},~\bfnm{Sean~P.}\binits{S.~P.}} \AND
  \bauthor{\bsnm{Tweedie},~\bfnm{R.~L.}\binits{R.~L.}}
(\byear{1993}).
\btitle{{Stability of {M}arkovian Processes III: {F}oster-{L}yapunov Criteria
  for Continuous-Time Processes}}.
\bjournal{{Applied Probability Trust}}
\bvolume{25}
\bpages{518-548}.
\end{barticle}
\endbibitem

\bibitem[\protect\citeauthoryear{Ogata}{1978}]{ogata2}
\begin{barticle}[author]
\bauthor{\bsnm{Ogata},~\bfnm{Yoshiko}\binits{Y.}}
(\byear{1978}).
\btitle{{The asymptotic behavior of maximum likelihood estimators for
  stationary point processes}}.
\bjournal{Annals of the Institute of Statistical Mathematics}
\bvolume{30}
\bpages{243-261}.
\bdoi{10.1007/BF02480216}
\end{barticle}
\endbibitem

\bibitem[\protect\citeauthoryear{Ogata}{1999}]{ogata}
\begin{barticle}[author]
\bauthor{\bsnm{Ogata},~\bfnm{Yosihiko}\binits{Y.}}
(\byear{1999}).
\btitle{Seismicity Analysis through Point-process Modeling: A Review}.
\bjournal{{P}ure and applied geophysics}
\bvolume{155}
\bpages{471-507}.
\bdoi{10.1007/s000240050275}
\end{barticle}
\endbibitem

\bibitem[\protect\citeauthoryear{Okatan, A~Wilson and N~Brown}{2005}]{neuro}
\begin{barticle}[author]
\bauthor{\bsnm{Okatan},~\bfnm{Murat}\binits{M.}},
  \bauthor{\bsnm{A~Wilson},~\bfnm{Matthew}\binits{M.}} \AND
  \bauthor{\bsnm{N~Brown},~\bfnm{Emery}\binits{E.}}
(\byear{2005}).
\btitle{Analyzing Functional Connectivity Using a Network Likelihood Model of
  Ensemble Neural Spiking Activity}.
\bjournal{Neural computation}
\bvolume{17}
\bpages{1927-61}.
\bdoi{10.1162/0899766054322973}
\end{barticle}
\endbibitem

\bibitem[\protect\citeauthoryear{Pillow, Wilson and Brown}{2008}]{pillow}
\begin{barticle}[author]
\bauthor{\bsnm{Pillow},~\bfnm{J.~W.}\binits{J.~W.}},
  \bauthor{\bsnm{Wilson},~\bfnm{M.~A.}\binits{M.~A.}} \AND
  \bauthor{\bsnm{Brown},~\bfnm{E.~N.}\binits{E.~N.}}
(\byear{2008}).
\btitle{{Spatio-temporal correlations and visual signalling in a complete
  neuronal population}}.
\bjournal{Nature}
\bvolume{454}
\bpages{995-999}.
\end{barticle}
\endbibitem

\bibitem[\protect\citeauthoryear{Raad}{2019}]{raad_renewal_2019}
\begin{barticle}[author]
\bauthor{\bsnm{Raad},~\bfnm{Mads~Bonde}\binits{M.~B.}}
(\byear{2019}).
\btitle{Renewal {Time} {Points} for {Hawkes} {Processes}}.
\bjournal{arXiv:1906.02036 [math]}.
\end{barticle}
\endbibitem

\bibitem[\protect\citeauthoryear{Reynaud-Bouret and
  Schbath}{2010}]{reynaudbouret}
\begin{barticle}[author]
\bauthor{\bsnm{Reynaud-Bouret},~\bfnm{Patricia}\binits{P.}} \AND
  \bauthor{\bsnm{Schbath},~\bfnm{Sophie}\binits{S.}}
(\byear{2010}).
\btitle{{Adaptive estimation for Hawkes processes; application to genome
  analysis}}.
\bjournal{{Ann. Statist.}}
\bvolume{38}
\bpages{2781--2822}.
\bdoi{10.1214/10-AOS806}
\end{barticle}
\endbibitem

\bibitem[\protect\citeauthoryear{Reynaud-Bouret et~al.}{2014}]{reynaudbouret2}
\begin{barticle}[author]
\bauthor{\bsnm{Reynaud-Bouret},~\bfnm{Patricia}\binits{P.}},
  \bauthor{\bsnm{Rivoirard},~\bfnm{Vincent}\binits{V.}},
  \bauthor{\bsnm{Grammont},~\bfnm{Franck}\binits{F.}} \AND
  \bauthor{\bsnm{Tuleau-Malot},~\bfnm{Christine}\binits{C.}}
(\byear{2014}).
\btitle{{Goodness-of-Fit Tests and Nonparametric Adaptive Estimation for Spike
  Train Analysis}}.
\bjournal{{Journal of Mathematical Neuroscience}}
\bvolume{4}
\bpages{3 - 330325}.
\bdoi{10.1109/TIT.1981.1056305}
\end{barticle}
\endbibitem

\bibitem[\protect\citeauthoryear{Y.~Kagan}{2009}]{kagan}
\begin{barticle}[author]
\bauthor{\bsnm{Y.~Kagan},~\bfnm{Yan}\binits{Y.}}
(\byear{2009}).
\btitle{{Statistical Distributions of Earthquake Numbers: Consequence of
  Branching Process}}.
\bjournal{{Geophys. J. Int.}}
\bvolume{180}.
\bdoi{10.1111/j.1365-246X.2009.04487.x}
\end{barticle}
\endbibitem

\bibitem[\protect\citeauthoryear{Zhou, Zha and Song}{2013}]{zhou}
\begin{bmisc}[author]
\bauthor{\bsnm{Zhou},~\bfnm{K.}\binits{K.}},
  \bauthor{\bsnm{Zha},~\bfnm{H.}\binits{H.}} \AND
  \bauthor{\bsnm{Song},~\bfnm{L.}\binits{L.}}
(\byear{2013}).
\btitle{{Learning triggering kernels for multi-dimensional Hawkes processes}}.
\bnote{{Proceedings of the 30th International Conf. on Machine Learning
  (ICML)}}.
\end{bmisc}
\endbibitem

\end{thebibliography}

\end{document}